\documentclass[amstex,11pt]{article}
\usepackage{amsmath,amsfonts,amssymb,amsthm,enumerate,hyperref,multicol,color,graphicx}

\usepackage[labelformat=brace]{subcaption}
\DeclareCaptionFormat{myformat}{\hspace*{1cm}#1}
\usepackage{amsmath}
\newtheorem{lemma}{Lemma}
\newtheorem{thm}{Theorem}
\newtheorem{definition}{Definition}

\newtheorem{proposition}{Proposition}

\usepackage[autostyle]{csquotes}

\numberwithin{equation}{section}

\begin{document}

\begin{center}
\Large{\textbf{Mass-conserving weak solutions to Oort-Hulst-Safronov coagulation equation with singular rates }}
\end{center}

\centerline{Prasanta Kumar Barik${}^1$, Pooja Rai${}^2$ and Ankik Kumar Giri${}^3$}\let\thefootnote\relax\footnotetext{Corresponding author. Tel +91-1332-284818 (O);  Fax: +91-1332-273560  \newline{\it{${}$ \hspace{.3cm} Email address: }}ankik.giri@ma.iitr.ac.in}  
\medskip
{\footnotesize

  \centerline{ ${}^{1}$Tata Institute of Fundamental Research, Centre for Applicable Mathematics,}
   \centerline{ Bangalore-560065, Karnataka, India}

  \centerline{ ${}^{2,3}$Department of Mathematics, Indian Institute of Technology Roorkee,}
   \centerline{ Roorkee-247667, Uttarakhand, India}

}

\bigskip

%
%

\begin{abstract}
Existence of global weak solutions to the continuous Oort-Hulst-Safronov (OHS) coagulation equation is investigated for coagulation kernels capturing a singularity near zero and growing linearly at infinity. The proof mainly relies on a  relation, between classical Smoluchowski coagulation equation (SCE) and OHS coagulation equation, which is introduced in \cite{Lachowicz:2003} as generalized coagulation equation. Moreover, all weak solutions formulated in a suitable sense are demonstrated to be mass-conserving. We obtain here a similar result for OHS coagulation equation as the one in \cite{Barik:2019} for SCE.
\end{abstract}



\section{Introduction}\label{existintroduction1}

Coagulation is a kinetic process which can be seen to take place in several physical circumstances. In this process, particles combine to form larger ones through coalescence where each particle is identified according to its size. A mathematical model to govern the coagulation of collides moving according to Brownian motion was first introduced by Smoluchowski \cite{Smoluchowski:1917}. The \emph{Smoluchowski coagulation model} was a discrete system of differential equations which illustrates the time evolution of particle size distribution for particles having  their sizes as positive integers. By considering the size of each particle  by a positive real number, the continuous version of the Smoluchowski coagulation model is given by M\"{u}ller \cite{Muller:1928}. Later, the discrete and the continuous Smoluchowski coagulation equations have been widely discussed by several researchers. The continuous Smoluchowski coagulation equation (SCE) describes the evolution of number density  $\zeta(\mu, t)$ for the particles of size $\mu \in \mathbb{R}_{>0}:=(0, \infty)$ at time $t \ge 0$ and expresses

\begin{align}\label{SCE}
\frac{\partial \zeta(\mu,t)}{\partial t}  = & \frac{1}{2} \int_0^{\mu} \Lambda(\mu-\nu, \nu) \zeta(\mu-\nu, t)   \zeta(\nu, t)d\nu \nonumber\\
& -\int_{0}^{\infty} \Lambda(\mu, \nu) \zeta(\mu, t) \zeta(\nu, t)d\nu,\ \qquad (\mu,t)\in \mathbb{R}_{>0}^2,
\end{align}
with initial condition
\begin{equation}\label{SCEin}
\zeta(\mu, 0) = \zeta^{in}(\mu)\ge 0,\ \qquad \mu\in \mathbb{R}_{>0}.
\end{equation}

The positive and negative terms on the right-hand side of \eqref{SCE} describe the birth and death of particles, respectively, as a result of
coagulation events. The coagulation rate, at which particles of size $\mu$ merge with particles of size $\nu$, is indicated by a non-negative symmetric function $\Lambda(\mu, \nu)$. Within a separate sense, a new coagulation model was introduced by Oort and van de Hulst in  \cite{Hulst:1946} and then further reformulated by Safronov in \cite{Safronov:1972}. Therefore, it is known as \emph{Oort-Hulst-Safronov (OHS) model}.
The corresponding discrete model is derived by Dubovski which is known as Safronov-Dubovski coagulation equation, see \cite{Davidson:2014, Dubovski:1999I, Dubovski:1999II, Bagland:2005} and references therein. The general Oort-Hulst-Safronov coagulation equation for the dynamics of evolution of the number density $\zeta(\mu, t)$ of particles of size $\mu>0$ at time $t \ge 0$ is given by
\begin{align}\label{OHS1}
\frac{\partial \zeta(\mu,t)}{\partial t}  = & -\frac{\partial }{\partial \mu} \bigg( \zeta(\mu, t) \int_0^{\mu} \nu \Lambda(\mu, \nu) \zeta(\nu, t)d\nu \bigg)\nonumber\\
& -\int_{\mu}^{\infty} \Lambda(\mu, \nu) \zeta(\mu, t) \zeta(\nu, t)d\nu,\ \qquad (\mu,t)\in \mathbb{R}_{>0}^2.\
\end{align}
 Here $\frac{\partial}{\partial t}$ and $\frac{\partial}{\partial \mu}$ represent the time and space partial derivatives, respectively. Furthermore, the non-negative and symmetric quantity $\Lambda(\mu, \nu)$ indicates the intensity rate at which particles of sizes $\mu$ and $\nu$ coagulate to form bigger ones. As mentioned in \cite{Lachowicz:2003}, in OHS coagulation model \eqref{OHS1}, the birth rate of particles of size $\mu$ from smaller particles does not depend on the sizes of the particles taken place in the coagulation process, but on a certain averaged quantity, and this is a basic difference with \eqref{SCE}. While the SCE and OHS models were developed in different ways, there is a relationship between these two models which has been established by Lachowicz et al.\ \cite{Lachowicz:2003} and known as \emph{generalized coagulation equation}. It should be mentioned that a relationship between these models was first observed by  Dubovski \cite{Dubovski:1999I} in discrete settings.
The purpose of the present work is to establish the existence of mass conserving weak solutions to the OHS coagulation equation \eqref{OHS1} with singular kernels via generalized coagulation equation. Thus, let us first revisit a generalized coagulation equation \cite{Lachowicz:2003} for the number density $\zeta(\mu, t)$ of particle of size $\mu$ at time $t$ as
\begin{align}\label{GenliCoagul}
\frac{\partial \zeta(\mu, t)}{\partial t} =&\frac{1}{2}\int_{0}^{\infty} \int_{0}^{\infty} P(\mu; \nu, \tau) \Lambda(\nu, \tau)  \zeta(\nu, t) \zeta(\tau, t) d\tau d\nu \nonumber\\
&- \int_{0}^{\infty} \Lambda(\mu, \nu) \zeta(\mu, t) \zeta(\nu, t)d\nu,
\end{align}
where $P$ is the nonnegative weighted probability function which describes the collision of a particle of size $\nu$ and a particle of size $\tau$ produces a particle of size $\mu$ and satisfying
\begin{align}\label{Weighted Prob1}
P(\mu; \nu, \tau) = P(\mu; \tau, \nu), \ \ \ (\mu, \nu, \tau) \in \mathbb{R}_{>0}^3,
\end{align}
and
\begin{align}\label{Weighted Prob2}
\int_0^{\infty} \mu P(\mu; \nu, \tau) d\mu = \nu+\tau, \ \ \ (\nu,\tau) \in  \mathbb{R}_{>0}^2.
\end{align}
Before reviewing the existing literature let us now define the total mass (volume) of the particles for the number density $\zeta$ at time $t\ge 0$ as:
\begin{equation}\label{TotalmassOHS}
\mathcal{M}_1(t)= \mathcal{M}_1(\zeta)(t):=\int_0^{\infty}{\mu} \zeta(\mu, t)d{\mu}.
\end{equation}

It is well understood, according to the conservation of matter, that the total mass of particles is neither produced nor destroyed. It is therefore anticipated that the total mass of the system will be preserved all over the time evolution specified by \eqref{SCE}--\eqref{SCEin}, that is, $\mathcal{M}_1(\zeta)(t)=\mathcal{M}_1(\zeta^{in})$ for each $t\ge 0$. It is worth noting, however, that for the multiplicative rate $\Lambda(\mu, \nu)=\mu \nu$, the total mass conservation breaks down for the SCE at finite time $t=1$, see \cite{Leyvraz:1981}. Physical understanding is that the missing mass corresponds to particles of infinite size produced by the uncontrolled growth in the system because of the extremely high rate of coalescence of really big particles. Such massive particles \cite{Aldous:1999} are called infinite-gels, and their appearance is called the  \emph{gelation phenomenon}. The earliest time at which the gelation phenomenon takes place is called the \emph{gelation time}.

The main objective of this paper is to have a global existence result on weak solutions to the OHS equation \eqref{OHS1}--\eqref{SCEin} for coagulation kernels having a small-volume singularity. Another task is to show that the weak solutions constructed here are indeed mass-conserving i.e. \eqref{ThemMasscons} holds. Now, let us review some of the current literature available for coagulation models, including OHS and SCE equations, and their associated models. For the last three decades, the SCE and its associated models have been of considerable interest to mathematicians and physicists.

 After the classical works of Stewart \cite{Stewart:1989, Stewart:1990} and Ball \& Carr \cite{Ball:1990}, there have been many papers dedicated to the existence and uniqueness of solutions to the SCE for coagulation kernels that are locally bounded, along with the mass conservation property and gelation phenomenon, see \cite{Aldous:1999, Barik:2018, Escobedo:2002, Escobedo:2003, Giri:2011, Giri:2012, Laurencot:2005, Laurencot:2006, Laurencot:2015, Laurencot:2002L, Mischler:2004, Norris:1999, Stewart:1990} and the references therein. Although all of these studies for singular coagulation kernels are of great mathematical significance,  due to the availability of physically relevant singular kernels such as: Smoluchowski coagulation kernel for the Brownian motion, granulation kernel and stochastic stirred froth \cite{Aldous:1999, Barik:2019, Norris:1999}, there are a few papers in which existence and uniqueness of solutions to the SCE with singular coagulation rates have been investigated, see \cite{Camejo:2015, Escobedo:2006, Norris:1999, Barik:2019}. Recently, in \cite{Barik:2019}, the global existence of weak solutions to the continuous SCE is shown for singular coagulation kernels which extends the results obtained in \cite{Camejo:2015, Norris:1999}. More specifically, a linear growth at infinity of the coagulation kernel is incorporated and the second moment of the initial data need not be finite. In addition, all weak solutions in a suitable sense are shown to be mass-conserving, a property that has been demonstrated in \cite{Norris:1999} under stronger assumptions. More details can be found in \cite{Barik:2019}.

Now, let us turn to the mathematical studies of the OHS coagulation model. In \cite{Lachowicz:2003}, after developing a relationship between the SCE and OHS model, Lachowicz et al.\ have established the existence of weak solutions to OHS coagulation equation \eqref{OHS1}--\eqref{SCEin} for a class of sub-additive coagulation kernel $\Lambda(\mu, \nu)\le k(1+\mu+\nu)$ and sub-quadratic coagulation kernel
$$\Lambda(\mu, \nu) \le k(1+\mu)(1+\nu)~\text{with}~ \sup_{\mu\in[0,R]}\frac{\Lambda(\mu, \nu)}{\nu}\rightarrow 0 ~\text{as}~ \nu\rightarrow +\infty$$
for each $R \geq 1$, where these coagulation kernels also satisfy $\Lambda \in \mathbb{W}_{\mbox{loc}}^{1,\infty}([0,\infty))^{2}$ and $\partial_{\mu}\Lambda(\mu, \nu)\geq-\alpha$ for some $\alpha\ge 0$. In addition, weak solutions to OHS coagulation equation \eqref{OHS1}--\eqref{SCEin} have been found to be mass-conserving for sub-additive kernels and the occurrence of gelation phenomenon takes place for sub-quadratic coagulation kernels. Furthermore, they have investigated the large time behavior of the solutions. It has also been observed that a compactly supported initial distribution propagates with finite speed which may not be possible for SCE. Recently, the global existence of solutions to the discrete version of OHS model i.e.\ Safronov-Dubovski coagulation equation has been investigated, by Davidson \cite{Davidson:2014}, for unbounded kernels. In addition, it is shown that the mass-conservation property is satisfied for kernels having  sub-linear growth at infinity. The questions of uniqueness and continuous dependence on the initial condition are also demonstrated for bounded kernels. In \cite{Laurencot:2005}, Lauren\c{c}ot addressed self-similar solutions to the OHS coagulation model with the coagulation kernel $\Lambda \equiv 1$. Moreover, it has also been shown that the existing solution is unique. In \cite{Laurencot:2006}, the existence of self-similar solutions is studied to the OHS model with the multiplicative coagulation kernel $\Lambda(\mu, \nu)=\mu\nu$.  Recently, in \cite{Bagland:2007}, Bagland and Lauren\c{c}ot have established the existence of self-similar solutions for the OHS model \eqref{OHS1}--\eqref{SCEin} when the coagulation kernel satisfies $\Lambda(\mu,\nu)=(\mu^{\lambda}+\nu^{\lambda})$, for $\lambda \in (0,1)$. However, this is clearly visible that the existing literature on the existence and uniqueness of solutions to the OHS model is quite limited in comparison to the classical SCE. Moreover, the existence of mass-conserving weak solutions to the OHS model for coagulation kernels having singularity for small volumes is missing. Therefore, the main novelty of the present work is to extend the result on the global existence of mass-conserving weak solutions to the OHS equation \eqref{OHS1}--\eqref{SCEin} in \cite{Lachowicz:2003} from non-singular coagulation kernels to singular coagulation kernels having linear growth at infinity. It is worth mentioning that the growth conditions \eqref{coagulation kernel} considered here on coagulation kernels and assumption on the the initial data are motivated from the ones considered in \cite{Barik:2019} for showing the existence of mass-conserving solutions to the SCE \eqref{SCE}--\eqref{SCEin} with singular rates. The existence proof is mainly motivated from \cite{Lachowicz:2003} and \cite{Barik:2019}.

Let us now outline the contents of the paper. Next section contains some assumptions on coagulation kernel and on the initial data together with the definition of weak solutions and the statement of the main result. In Section 3, the existence of a weak solution to the OHS equation \eqref{OHS1}--\eqref{SCEin} is shown by using a classical weak $L^1$ compactness method. In the last section, it is investigated that all weak solutions defined in a suitable sense are mass-conserving.



\section{Assumptions, Definition and Main result}
 In order to state the main existence result to the OHS model \eqref{OHS1}--\eqref{SCEin}, let us first introduce some assumptions on the coagulation kernel and initial data.\\ \\
Assume that the coagulation kernel $\Lambda$ is a non-negative measurable function on $\mathbb{R}_{>0} \times \mathbb{R}_{>0}$ and there are $\sigma\ge 0$ and $k \ge 0$ such that
\begin{equation}\label{coagulation kernel}
\left.
\begin{split}
0 \le \Lambda(\mu, \nu) & = \Lambda(\nu, \mu) \le k (\mu \nu)^{-\sigma},\  (\mu, \nu)\in (0, 1)^2, \\
0 \le \Lambda(\mu, \nu) &=\Lambda(\nu, \mu) \le k \mu \nu^{-\sigma},\  (\mu, \nu) \in [1, \infty) \times (0, 1),\\
0 \le \Lambda(\mu, \nu) &=  \Lambda(\nu, \mu) \le k (\mu+\nu),\  (\mu, \nu)\in [1, \infty)^2,
\end{split}
\right\}
\end{equation}
and the coagulation kernel $\Lambda$ also satisfies
\begin{align}\label{bound coagulation kernel}
\Lambda \in W_{\mbox{loc}}^{1, \infty}([0, \infty)^2),  \ \mbox{and} \ \Lambda_{\mu} (\mu, \nu)  \ge -\eta \mu^{-\sigma-1}\nu^{-\sigma},
\end{align}
for some $\eta \geq 0$ and $(\mu, \nu) \in \mathbb{R}_{>0}^2$.

Next, let us assume that the initial data $\zeta^{in}$ belongs to the following weighted $L^1$ space
\begin{align}
0 \le \zeta^{in} \in \mathcal{Y}^{+},
\end{align}
where $\mathcal{Y}^{+}$ denotes the positive cone of the Banach space
\begin{equation}\label{Banachspace}
\mathcal{Y}=L^{1}(\mathbb{R}_{>0};\ (\mu+\mu^{-2\sigma})d\mu),
\end{equation}
endowed with the norm
\begin{equation}\label{Banachspace1}
\|\zeta\|_{\mathcal{Y}}=\int_{0}^{\infty}(\mu+\mu^{-2\sigma})|\zeta(\mu,t)|d\mu.
\end{equation}

\begin{definition}\label{definition}
Assume that $\zeta^{in} \in \mathcal{Y}^{+}$ and $T\in(0, \infty]$. For $t \in[0, T)$, a weak solution to \eqref{OHS1}--\eqref{SCEin} on $[0, T)$ is a nonnegative function
\begin{equation*}
\zeta  \in \mathcal{C}_w([0,T); L^{1}(\mathbb{R}_{>0}; d\mu))\cap L^{\infty}(0,T; L^1(\mathbb{R}_{>0}; (\mu^{-\sigma} + \mu) d\mu ) ),
\end{equation*}
 satisfying
\begin{align}\label{weaksolohs}
\int_{0}^{\infty} \omega(\mu) \{ \zeta(\mu,t)- \zeta^{in}(\mu) \}d\mu
=\int_{0}^{t}\int_{0}^{\infty}\int_{0}^{\nu} \omega_1(\nu, \tau)  \Lambda(\nu, \tau)  \zeta(\nu, s) \zeta(\tau, s) d\tau d\nu ds,
\end{align}
where
\begin{align}\label{Identity1}
 \omega_1(\nu, \tau) = \tau \omega'(\nu)-\omega(\tau),
 \end{align}
 for any $\omega\in\mathbb{W}^{1,\infty}(\mathbb{R}_{>0})$, and  $\omega'$ is compactly supported. The space $\mathcal{C}_w([c, d]; L^{1}(\mathbb{R}_{>0}; d\mu))$ indicates the space of all weakly continuous functions from $[c, d]$ to $L^{1}(\mathbb{R}_{>0}; d\mu)$.
\end{definition}


Next, we state the weak formulations of other two coagulation models.
For any test function $\omega \in \mathcal{C}_c^{\infty}(\mathbb{R}_{>0})$, the weak formulation to the classical SCE \eqref{SCE}--\eqref{SCEin} is given by

\begin{align}\label{WeakSCM}
\int_{0}^{\infty} \omega(\mu) \{ \zeta(\mu, t)-\zeta^{in}(\mu) \} d\mu = & \int_0^t  \int_{0}^{\infty} \int_{0}^{\tau}
 \tilde{\omega}(\nu, \tau)  \Lambda(\nu, \tau) \zeta(\nu, s) \zeta(\tau, s) d\nu d\tau ds,
\end{align}
where
\begin{align}\label{Identity2}
 \tilde{\omega}(\nu, \tau) =  \omega(\nu+\tau)-\omega(\nu)-\omega(\tau),
 \end{align}
 and the weak formulation to \eqref{GenliCoagul}--\eqref{SCEin} is stated as
\begin{align}\label{WeakGen}
 \int_{0}^{\infty} \omega(\mu) \{ \zeta(\mu, t) -\zeta^{in}(\mu) \} d\mu =  \int_0^t  \int_{0}^{\infty} \int_{0}^{\nu} \omega_2(\nu, \tau) \Lambda(\nu, \tau) \zeta(\nu, s) \zeta(\tau, s) d\tau d\nu ds,
\end{align}
where
\begin{align}\label{Identity3}
 \omega_2(\nu, \tau) =  \int_{0}^{\infty} P(\mu; \nu, \tau) \omega(\mu)d\mu -\omega(\nu)-\omega(\tau).
 \end{align}

 In order to obtain the common frame of both coagulation equations \eqref{SCE} and \eqref{OHS1} by a family of generalized coagulation equation \eqref{GenliCoagul}, we replace $P_{\varepsilon}$ for $P$ and $\Lambda_{\varepsilon}$ for $\Lambda$ and substituting them into \eqref{WeakGen}, where
\begin{align}\label{Weighted Prob11}
P_{\varepsilon}(\mu; \nu, \tau)=\delta(\mu-\max
\{\nu, \tau\}-\varepsilon\min\{\nu, \tau \}) + (1-\varepsilon)\delta(\mu-\min\{\nu, \tau\}),
\end{align}
and
\begin{align}\label{aGenerli}
\Lambda_{\varepsilon}(\nu, \tau)=\frac{\Lambda(\nu, \tau)}{\varepsilon},
\end{align}
for $\varepsilon\in(0, 1]$, hence, we obtain
\begin{align}\label{WeakGenSCEli}
\int_{0}^{\infty} \omega(\mu) [\zeta_{\varepsilon}(\mu, t) -\zeta^{in}_{\varepsilon}(\mu) ] d\mu = \int_0^t \int_{0}^{\infty} \int_{0}^{\nu}  \omega_{\varepsilon}(\nu, \tau)  \Lambda(\nu, \tau) \zeta_{\varepsilon}(\nu, s) \zeta_{\varepsilon}(\tau, s) d\tau d\nu ds,
\end{align}
where
\begin{align}\label{Identity4}
 \omega_{\varepsilon}(\nu, \tau) = \frac{\omega(\nu+\varepsilon \tau)-\omega(\nu)}{\varepsilon}-\omega(\tau).
 \end{align}

Analogously, for $\varepsilon\in(0,1]$, a non-negative function
 \begin{equation}\label{Weaksolugenera}
 \zeta_{\varepsilon}\in \mathcal{C}_w([0, T); L^{1}(\mathbb{R}_{>0}; d\mu))\cap L^{\infty}(0,T;  L^1(\mathbb{R}_{>0}; (\mu^{-\sigma} + \mu) d\mu ))
 \end{equation}
 is a weak solution to the following equation
 \begin{align}\label{GenerSCEeps}
\partial_{t}(\zeta_{\varepsilon})(\mu, t)= Q(\zeta_{\varepsilon})(\mu, t),\ \ \ (\mu, t) \in \mathbb{R}_{>0}^{2},
\end{align}
with
\begin{equation}\label{GenerSCEepsin}
 \zeta_{\varepsilon}(\mu, 0)=\zeta^{in}(\mu) \ge 0,
 \end{equation}
where
\begin{align*}
Q(\zeta_{\varepsilon})(\mu, t)=&\frac{1}{\varepsilon}\int_{0}^{\mu/(1+\varepsilon)} \Lambda(\mu-\varepsilon \nu, \nu) \zeta_{\varepsilon}(\mu-\varepsilon \nu, t) \zeta_{\varepsilon}(\nu, t)d\nu\\
&-\zeta_{\varepsilon}(\mu, t)\bigg[\bigg(\frac{1}{\varepsilon}-1\bigg)\int_{0}^{\mu}\Lambda(\mu, \nu) \zeta_{\varepsilon}(\nu, t)d\nu+\int_{0}^{\infty}\Lambda(\mu, \nu) \zeta_{\varepsilon}(\nu, t) d\nu \bigg].
\end{align*}

 For $\varepsilon=1$ and $\varepsilon\rightarrow 0$, one can see that the equation \eqref{WeakGenSCEli} converges to the weak formulation of \eqref{SCE} and \eqref{OHS1}, respectively. Here, we will indeed prove the convergence of the solution $\zeta_{\varepsilon}$ to \eqref{GenerSCEeps},\eqref{GenerSCEepsin} towards a weak solution to \eqref{OHS1}--\eqref{SCE} as $\varepsilon \to 0$.\\

This is the right time to state our main result of this paper.
\begin{thm}\label{Theorem1}
Let the coagulation kernel satisfies \eqref{coagulation kernel}--\eqref{bound coagulation kernel} and the initial data $\zeta^{in} \in \mathcal{Y}^+$. Then, for $\varepsilon \in (0, 1]$, there exists a subsequence $(\zeta_{\varepsilon_k})$ of $(\zeta_{\varepsilon})$  to \eqref{GenerSCEeps}--\eqref{GenerSCEepsin}  and a weak solution $\zeta$ to \eqref{OHS1}--\eqref{SCEin} on $[0, \infty)$ which satisfies
\begin{align}\label{ThemExistence}
\zeta_{\varepsilon_k} \to \zeta,\ \ \text{in}\ \ \mathcal{C}_w([0, T); L^1(\mathbb{R}_{>0};(\mu^{-\sigma} + \mu) d\mu))
\end{align}
as $\varepsilon_k \to 0$, for each $T \in (0, \infty]$ and $t \in [0, T)$. Furthermore, $\zeta$ satisfies
\begin{align}\label{ThemMasscons}
\mathcal{M}_1(\zeta^{in}):= \int_0^{\infty} \mu \zeta^{in}(\mu) d\mu= \int_0^{\infty} \mu \zeta(\mu, t) d\mu= \mathcal{M}_1(\zeta)(t), \ \ t\ge 0.
\end{align}
\end{thm}


\section{Existence of weak solutions}

Proof of the existence result is based on a weak compactness technique in $L^1(\mathbb{R}_{>0}; d\mu)$ which was first proposed by Stewart in \cite{Stewart:1989} for the continuous SCE and further elaborated in later articles. For the OHS model, this technique was well implemented in \cite{Lachowicz:2003}.  In order to apply the weak $L^1$ compactness technique, one has to construct a sequence of solutions to some approximating equations that will be done here later.\\

 Now, let us first recall a particular type of convex functions and their properties. As $\zeta^{in}\in \mathcal{Y}^+$, a refined version of de la Vall\'{e}e-Poussin theorem is given in  \cite[Theorem~2.8]{Laurencot:2015} guarantees the existence of two non-negative convex functions $\Psi_1, \Psi_2 \in \mathcal{C}^{2}([0, \infty))$, whose derivatives are concave, satisfying
\begin{align}\label{convexp1}
\Psi_j(0)=0,~~~\lim_{s \to {\infty}}\frac{\Psi_j(s)}{s}=\infty,
\end{align}
and
\begin{align}\label{convexp2}
\Gamma_1 := \int_0^{\infty} \Psi_1(\mu) \zeta^{in}(\mu)d\mu<\infty,~~\text{and}~~\Gamma_2 :=\int_0^{\infty}{\Psi_2(\mu^{-\sigma} \zeta^{in}(\mu))}d\mu<\infty.
\end{align}
The above convex functions obey the following properties.

\begin{lemma}\label{Lemmaconvex}
Let us consider $\Psi_1$ and $\Psi_2$ be two convex functions and their derivatives are concave functions. Then, the following results hold 
\begin{equation}\label{convexp3}
\hspace{-5cm} \Psi_2(z_1)\le z_1 \Psi'_2(z_1)\le 2\Psi_2(z_1),
\end{equation}
\begin{equation}\label{convexp4}
\hspace{-5.5cm} z_1\Psi'_2(z_2)\le \Psi_2(z_1)+\Psi_2(z_2),
\end{equation}
and
\begin{equation}\label{convexp5}
0 \le \Psi_1(z_1+z_2)-\Psi_1(z_1)-\Psi_1(z_2)\le  2\frac{z_1\Psi_1(z_2)+z_2\Psi_1(z_1)}{(z_1+z_2)},
\end{equation}
 for all $z_1, z_2 \in \mathbb{R}_{>0}$. 
\end{lemma}

Next, we introduce the following convex function
\[
 \Psi_{j,\lambda}(\mu) :=
\begin{cases}\label{convexlambda}
\Psi_j(\mu) & \text{if $\mu \in [0, \lambda ]$} \\
\Psi'_j(\lambda)(\mu-\lambda)+\Psi_j(\lambda) & \text{if $ \mu \in [\lambda, \infty)$},
\end{cases}
\]
for $\lambda\geq2$ and $j=1, 2$. One can easily check that the above $\Psi_{j,\lambda}$, for $j=1,2$ satisfy \eqref{convexp1}--\eqref{convexp2} and Lemma \ref{Lemmaconvex}. Now, this is the right time to derive a priori estimates in the next subsection.

\subsection{Moment Estimate}
\begin{lemma}\label{LemmaUniformbound1}
Let $T>0$ and $\epsilon \in (0, 1]$. Assume that $ \zeta^{in } \in  \mathcal{Y}^+$ and the coagulation kernel satisfies \eqref{coagulation kernel}. Let $\zeta_{\epsilon } \in \mathcal{C}([0, T); L^1(\mathbb{R}_{>0}; d\mu))$ be a weak solution \eqref{GenerSCEeps}--\eqref{GenerSCEepsin} on $[0, T)$. Then, we have
\begin{align}\label{UniformInt2}
 \int_{0}^{\infty} ( \mu^{-2\sigma}+\mu) \zeta_{\varepsilon}(\mu, t)  d\mu \le \Theta,
\end{align}
where $\Theta :=  \int_{0}^{\infty} (\mu^{-2\sigma}+\mu) \zeta^{in}(\mu)d\mu$.
\end{lemma}

\begin{proof}
Consider $\omega(\mu) :=(\mu+\delta)^{-2\sigma}$ for $\delta \in (0, 1)$ and insert it into \eqref{WeakGenSCEli}, we obtain
\begin{align}\label{Unibound1}
\int_{0}^{\infty}  (\mu+\delta)^{-2\sigma}  \{ \zeta_{\varepsilon}(\mu,t)- \zeta_{\varepsilon}^{in}(\mu) \} d\mu = \int_0^t \int_{0}^{\infty} \int_{0}^{\nu}  \omega_{\varepsilon}(\nu, \tau)  \Lambda(\nu, \tau) \zeta_{\varepsilon}(\nu, s) \zeta_{\varepsilon}(\tau, s) d\tau d\nu ds.
\end{align}
Next, we evaluate $\omega_{\varepsilon}$ as
\begin{align*}
\omega_{\varepsilon} (\nu, \tau) = & \frac{(\nu+\varepsilon \tau+\delta)^{-2\sigma}-(\nu+\delta)^{-2\sigma}}{\varepsilon}-(\tau+\delta)^{-2\sigma} \nonumber\\
\le &\frac{ (\nu+\delta)^{-2\sigma}-(\nu+\delta)^{-2\sigma}}{\varepsilon}-(\tau+\delta)^{-2\sigma} \le 0.
\end{align*}
Applying above inequality into \eqref{Unibound1}, we get
\begin{align}\label{Unibound2}
\int_{0}^{\infty}(\mu+\delta)^{-2\sigma} \zeta_{\varepsilon}(\mu,t)d\mu \le \int_{0}^{\infty}(\mu+\delta)^{-2\sigma} \zeta_{\varepsilon}^{in}(\mu)d\mu.
\end{align}
For $\delta \to 0$, an application of Fatou's lemma to \eqref{Unibound2} gives
\begin{align}\label{Unibound3}
\int_{0}^{\infty} \mu^{-2\sigma} \zeta_{\varepsilon}(\mu, t) d\mu \le \int_{0}^{\infty} \mu^{-2\sigma} \zeta^{in}(\mu)d\mu.
\end{align}
Similarly, setting $\omega(\mu) :=\mu \wedge  \lambda$, for $\lambda >1$ into \eqref{WeakGenSCEli}, we obtain
\begin{align}\label{Unibound4}
\int_{0}^{\infty} & (\mu \wedge \lambda )  \zeta_{\varepsilon}(\mu,t)d\mu = \int_{0}^{\infty} (\mu \wedge \lambda ) \zeta_{\varepsilon}^{in}(\mu)d\mu\nonumber\\
&+ \int_{0}^{t} \int_{0}^{\infty} \int_{0}^{\nu} \omega_{\varepsilon}(\nu, \tau)  \Lambda(\nu, \tau) \zeta_{\varepsilon}(\nu, s) \zeta_{\varepsilon}(\tau, s) d\tau d\nu ds.
\end{align}

Now, for $\tau \in (0, \nu)$, we determine $\omega_{\varepsilon}$ as
\[
\omega_{\varepsilon}(\nu, \tau) =
\begin{cases}\label{convexlambda1}
0, &  \text{if $\nu \in (0, \lambda)$, $\nu +\varepsilon \tau \in (0, \lambda)$}, \\
\frac{ \lambda-\nu }{\varepsilon}- \tau,    & \text{if $\nu \in (0, \lambda)$, $\nu +\varepsilon \tau \in [\lambda, \infty)$}, \\
-(\tau \wedge \lambda) \le 0, & \text{if $ \nu \in [\lambda, \infty)$}.
\end{cases}
\]
From \eqref{Unibound4} and $\omega_{\varepsilon} \le 0$, we thus get
\begin{align}\label{Unibound5}
\int_{0}^{\infty}  (\mu \wedge \lambda )  \zeta_{\varepsilon}(\mu,t)d\mu \le \int_{0}^{\infty} (\mu \wedge \lambda ) \zeta_{\varepsilon}^{in}(\mu)d\mu.
\end{align}

Again, using the Fatou lemma for $\lambda \rightarrow \infty$, we get
\begin{align}\label{Unibound5}
\int_{0}^{\infty} \mu \zeta_{\varepsilon}(\mu, t)d\mu \le \int_{0}^{\infty} \mu \zeta^{in}(\mu)d\mu.
\end{align}
Combining \eqref{Unibound3} and \eqref{Unibound5}, we obtain the required result. This completes the proof of Lemma \ref{LemmaUniformbound1}.
\end{proof}


\begin{lemma}\label{Lemmalargevalue}
Let $T>0$ and $\epsilon \in (0, 1]$. Assume that $ \zeta^{in } \in  \mathcal{Y}^+$ and \eqref{coagulation kernel} holds. Let $\zeta_{\epsilon } \in \mathcal{C}([0, T); L^1(\mathbb{R}_{>0}; d\mu))$ be a weak solution \eqref{GenerSCEeps}--\eqref{GenerSCEepsin} on $[0, T)$. Then
\begin{align}\label{Lemmalargevalue1}
\sup_{t \in [0, T) }  \int_0^{\infty} \Psi_1(\mu) \zeta_{\epsilon} (\mu, t) d\mu \le \Theta_2(T),
\end{align}
where $\Theta_2(T):=\bigg( \Gamma_1 +6 kT \Psi_{j,\lambda}(1) \Theta^2  \bigg) \exp(6 T k \Theta)$ and $\Psi_1$ satisfies \eqref{convexp1}, \eqref{convexp2}, \eqref{convexp3} and \eqref{convexp4}.
\end{lemma}

\begin{proof}
Setting $\omega := \Psi_{j,\lambda} $ into \eqref{WeakGenSCEli}. Since $\Psi_{j,\lambda}$ is a convex function and $\varepsilon \in (0, 1]$, then by the definition of the convex function, we have
\begin{align}\label{Large1}
\Psi_{j,\lambda}(\mu+\varepsilon \nu)-\Psi_{j,\lambda}(\mu) = & \Psi_{j,\lambda}((1-\varepsilon) \mu+(\mu+\nu)\varepsilon )-\Psi_{j,\lambda}(\mu)\nonumber\\
\le & (1-\varepsilon)\Psi_{j,\lambda}(\mu) + \varepsilon \Psi_{j,\lambda}(\mu+\nu) -\Psi_{j,\lambda}(\mu) \nonumber\\
 = & \varepsilon  \{ \Psi_{j,\lambda}(\mu+\nu) - \Psi_{j,\lambda}(\mu) \},
\end{align}
for every $(\mu, \nu) \in \mathbb{R}_{>0} \times \mathbb{R}_{>0}$. Next, using \eqref{Large1} and \eqref{convexp5} into \eqref{WeakGenSCEli}, we obtain
\begin{align}\label{Large2}
\int_0^{\infty} \Psi_{j,\lambda}(\mu) & \zeta_{\varepsilon}(\mu, t) d\mu -\int_0^{\infty} \Psi_{j,\lambda}(\mu) \zeta^{in}_{\varepsilon}(\mu) d\mu \nonumber\\
= & \int_0^t \int_0^{\infty} \int_0^\mu \bigg\{ \frac{\Psi_{j,\lambda}(\mu+\varepsilon \nu)-\Psi_{j,\lambda}(\mu)}{\varepsilon}-\Psi_{j,\lambda}(\nu) \bigg\} \Lambda(\mu, \nu) \zeta_{\varepsilon}(\mu, s) \zeta_{\varepsilon}(\nu, s) d\nu d\mu ds \nonumber\\
\le &  \int_0^t \int_0^{\infty} \int_0^\mu \{\Psi_{j,\lambda}(\mu+\nu) - \Psi_{j,\lambda}(\mu) -\Psi_{j,\lambda}(\nu) \} \Lambda(\mu, \nu) \zeta_{\varepsilon}(\mu, s) \zeta_{\varepsilon}(\nu, s) d\nu d\mu ds \nonumber\\
\le & 2  \int_0^t \int_0^{\infty} \int_0^\mu \frac{ \{ \nu \Psi_{j,\lambda}(\mu) + \mu \Psi_{j,\lambda}(\nu)\} }{(\mu+\nu)} \Lambda(\mu, \nu) \zeta_{\varepsilon}(\mu, s) \zeta_{\varepsilon}(\nu, s) d\nu d\mu ds \nonumber\\
= & 2  \int_0^t \int_0^{1} \int_0^\mu \frac{ \{ \nu \Psi_{j,\lambda}(\mu) + \mu \Psi_{j,\lambda}(\nu) \} }{(\mu+\nu)} \Lambda(\mu, \nu) \zeta_{\varepsilon}(\mu, s) \zeta_{\varepsilon}(\nu, s) d\nu d\mu ds \nonumber\\
& + 2  \int_0^t \int_1^{\infty} \int_0^1\frac{ \{ \nu \Psi_{j,\lambda}(\mu) + \mu \Psi_{j,\lambda}(\nu) \} }{(\mu+\nu)} \Lambda(\mu, \nu) \zeta_{\varepsilon}(\mu, s) \zeta_{\varepsilon}(\nu, s) d\nu d\mu ds \nonumber\\
& + 2 \int_0^t \int_1^{\infty} \int_1^\mu \frac{ \{ \nu\Psi_{j,\lambda}(\mu) + \mu \Psi_{j,\lambda}(\nu) \} }{(\mu+\nu)} \Lambda(\mu, \nu) \zeta_{\varepsilon}(\mu, s) \zeta_{\varepsilon}(\nu, s) d\nu d\mu ds.
\end{align}
 Let us first evaluate the first integral on the right-hand side to \eqref{Large2}, by using \eqref{coagulation kernel} and Lemma \ref{LemmaUniformbound1}, as
\begin{align}\label{Large3}
 2 k \int_0^t \int_0^{1} \int_0^\mu & \frac{ \{ \nu \Psi_{j,\lambda}(\mu) + \mu \Psi_{j,\lambda}(\nu) \} }{(\mu+\nu)} (\mu \nu)^{-\sigma}  \zeta_{\varepsilon}(\mu, s) \zeta_{\varepsilon}(\nu, s) d\nu d\mu ds \nonumber\\
 \le & 4 k \int_0^t \int_0^{1} \int_0^\mu \frac{ \mu \Psi_{j,\lambda}(\mu) }{(\mu+\nu)} (\mu \nu)^{-\sigma}  \zeta_{\varepsilon}(\mu, s) \zeta_{\varepsilon}(\nu, s) d\nu d\mu ds \nonumber\\
 \le & 4 k \Psi_{j,\lambda}(1) \int_0^t \int_0^{1} \int_0^1   (\mu \nu)^{-\sigma}  \zeta_{\varepsilon}(\mu, s) \zeta_{\varepsilon}(\nu, s) d\nu d\mu ds \le  4 k   \Theta^2 \Psi_{j,\lambda}(1)T.
\end{align}
Similarly, the second integral on the right-hand side to \eqref{Large2} can be estimated as
\begin{align}\label{Large4}
 2 k \int_0^t  \int_1^{\infty} \int_0^1 & \frac{ \{ \nu \Psi_{j,\lambda}(\mu) + \mu \Psi_{j,\lambda}(\nu) \} }{(\mu+\nu)} \mu \nu^{-\sigma} \zeta_{\varepsilon}(\mu, s) \zeta_{\varepsilon}(\nu, s) d\nu d\mu ds \nonumber\\
 \le &  2 k \int_0^t \int_1^{\infty} \int_0^1  \frac{ \{  \Psi_{j,\lambda}(\mu) + \mu \Psi_{j,\lambda}(1) \} }{(\mu+\nu)} \mu \nu^{-\sigma} \zeta_{\varepsilon}(\mu, s) \zeta_{\varepsilon}(\nu, s) d\nu d\mu ds \nonumber\\
 \le &  2 k  \Theta \int_0^t \int_0^{\infty}  \Psi_{j,\lambda}(\mu)  \zeta_{\varepsilon}(\mu, s) d\mu  ds+  2 k \Theta^2 \Psi_{j,\lambda}(1)T.
\end{align}
Again thanks to \eqref{coagulation kernel} and Lemma \ref{LemmaUniformbound1} for further calculating the last integral on the right-hand side to \eqref{Large2} as
\begin{align}\label{Large5}
 2 k & \int_0^t \int_1^{\infty} \int_1^\mu  \{ \nu \Psi_{j,\lambda}(\mu) + \mu \Psi_{j,\lambda}(\nu) \} \zeta_{\varepsilon}(\mu, s) \zeta_{\varepsilon}(\nu, s) d\nu d\mu ds \nonumber\\
\le &  2 k \int_0^t \int_1^{\infty} \int_1^{\infty} \nu \Psi_{j,\lambda}(\mu)\zeta_{\varepsilon}(\mu, s) \zeta_{\varepsilon}(\nu, s) d\nu d\mu ds
\le   2 k \Theta \int_0^t \int_0^{\infty}  \Psi_{j,\lambda}(\mu)  \zeta_{\varepsilon}(\mu, s)  d\mu ds.
\end{align}

Now, substituting \eqref{Large3}, \eqref{Large4} and \eqref{Large5} into \eqref{Large2} and then applying Gronwall's inequality, we obtain
\begin{align}\label{Large6}
\int_0^{\infty}  \Psi_{j,\lambda}(\mu)  \zeta_{\varepsilon}(\mu, t)  d\mu \le \Theta_2(T).
\end{align}
By taking $j=1$ and $\lambda \to \infty$ in \eqref{Large6},  we complete the proof of Lemma \ref{Lemmalargevalue}.
\end{proof}


\subsection{Uniform Integrability}
\begin{lemma}\label{LemmaUniformInt1}
Let $T>0$ and $\epsilon \in (0, 1]$. Assume that $ \zeta^{in } \in  \mathcal{Y}^+$ and the coagulation kernel satisfies \eqref{coagulation kernel}--\eqref{bound coagulation kernel}. Let $\zeta_{\epsilon } \in \mathcal{C}([0, T); L^1(\mathbb{R}_{>0}; d\mu))$ be a weak solution \eqref{GenerSCEeps}--\eqref{GenerSCEepsin} on $[0, T)$. Then the following inequality holds true
\begin{align}\label{UniformInt2}
\sup_{t \in  [0, T)}\int_{0}^{\infty} \Psi_2( \mu^{-\sigma} \zeta_{\varepsilon}(\mu, t)) d\mu \le \Theta_{1}(T),
\end{align}
where $\Theta_1(T):= \Gamma_2 \exp(k T  \Theta )$ and $\Psi_2$ satisfies \eqref{convexp1}--\eqref{convexp4}.
\end{lemma}

\begin{proof}
Set $h_{\varepsilon}(\mu,t)= \mu^{-\sigma} \zeta_{\varepsilon}(\mu,t)$. For $j=2$, we infer from \eqref{GenerSCEeps} that
\begin{align}\label{UniInteg1}
\frac{d}{dt}\int_{0}^{\infty} \Psi_{2,\lambda} & (h_{\varepsilon} (\mu, t)) d\mu =\int_{0}^{\infty} \mu^{-\sigma} \Psi'_{2,\lambda}(h_{\varepsilon}(\mu, t))
\frac{\partial \zeta_{\varepsilon}(\mu, t)}{\partial t}d\mu \nonumber\\
=&\int_{0}^{\infty} \mu^{-\sigma} \Psi'_{2,\lambda}(h_{\varepsilon}(\mu, t)) \bigg\{ \frac{1}{\varepsilon}\int_{0}^{\mu/(1+\varepsilon)} \Lambda(\mu-\varepsilon \nu, \nu) \zeta_{\varepsilon}(\mu-\varepsilon \nu, t) \zeta_{\varepsilon}(\nu, t) d\nu \nonumber\\
&-\zeta_{\varepsilon}(\mu, t)\bigg[\bigg(\frac{1}{\varepsilon}-1\bigg)\int_{0}^{\mu}\Lambda(\mu, \nu)\zeta_{\varepsilon}(\nu, t)d\nu
+\int_{0}^{\infty}\Lambda(\mu, \nu)\zeta_{\varepsilon}(\nu, t) d\nu \bigg] \bigg\}d\mu.
\end{align}
By applying repeated applications of Fubini's theorem and using $\mu-\varepsilon \nu=\mu'$ and $\nu=\nu'$ to the first integral on the right-hand side to \eqref{UniInteg1}, we obtain
\begin{align}\label{UniInteg2}
\frac{d}{dt}\int_{0}^{\infty} \Psi_{2, \lambda}(h_{\varepsilon}(\mu, t)) d\mu
=& \frac{1}{\varepsilon} \int_{0}^{\infty}  \int_{0}^{\mu}  (\mu+ \varepsilon \nu)^{-\sigma}  \Psi'_{2, \lambda}(h_{\varepsilon}(\mu +\varepsilon \nu, t))  \Lambda(\mu, \nu) \zeta_{\varepsilon}(\mu, t) \zeta_{\varepsilon}(\nu, t) d\nu d\mu \nonumber\\
&-\bigg(\frac{1}{\varepsilon}-1\bigg) \int_{0}^{\infty} \int_{0}^{\mu}   \mu^{-\sigma} \Psi'_{2,\lambda}(h_{\varepsilon}(\mu, t)) \Lambda(\mu, \nu) \zeta_{\varepsilon}(\mu, t) \zeta_{\varepsilon}(\nu, t) d\nu d\mu \nonumber\\
& - \int_{0}^{\infty} \int_{0}^{\infty}  \mu^{-\sigma} \Psi'_{2,\lambda}(h_{\varepsilon}(\mu, t)) \Lambda(\mu, \nu)  \zeta_{\varepsilon}(\mu, t)  \zeta_{\varepsilon}(\nu, t) d\nu d\mu.
\end{align}
Since $\zeta_{\varepsilon}$ and $\Lambda$ are non-negative functions, thus from \eqref{UniInteg2}, we get
\begin{align}\label{UniInteg3}
\frac{d}{dt}&\int_{0}^{\infty} \Psi_{2,\lambda}  (h_{\varepsilon}(\mu, t)) d\mu\nonumber\\
\le & \frac{1}{\varepsilon} \int_{0}^{\infty}  \int_{0}^{\mu}  \mu^{-\sigma} \{ \Psi'_{2, \lambda}(h_{\varepsilon}(\mu +\varepsilon \nu, t)) -\Psi'_{2, \lambda} (h_{\varepsilon}(\mu, t)) \}  \Lambda(\mu, \nu)  \zeta_{\varepsilon}(\mu, t) \zeta_{\varepsilon}(\nu, t) d\nu d\mu.
\end{align}
 Due to the convexity of $\Psi_{2, \lambda}$, one can infer that $\mu (\Psi'_{2, \lambda}(\nu)-\Psi'_{2, \lambda}(\mu)) \le \Phi_{2, \lambda}(\nu)-\Phi_{2, \lambda}(\mu)$ for $\mu,\nu\ge 0$, where $\Phi_{2, \lambda}(\mu)=\mu \Psi'_{2, \lambda}(\mu)-\Psi_{2, \lambda}(\mu),~\mu\ge 0$. Now, using this argument into \eqref{UniInteg3}, we obtain
\begin{align}\label{UniInteg4}
\frac{d}{dt}\int_{0}^{\infty} & \Psi_{2, \lambda}(h_{\varepsilon}(\mu, t)) d\mu \nonumber\\
\le &\frac{1}{\varepsilon}\int_{0}^{\infty}\int_{0}^{\mu}
 \{ \Phi_{2, \lambda}(h_{\varepsilon}(\mu +\varepsilon \nu, t))-\Phi_{2, \lambda}(h_{\varepsilon}(\mu, t)) \}\Lambda(\mu, \nu)  \zeta_{\varepsilon}(\nu, t) d\nu d\mu.
\end{align}
Again changing the order of integration by Fubini's theorem and using the transformation $\mu +\varepsilon \nu= \mu'$ and $\nu=\nu'$ into \eqref{UniInteg4}, we have
\begin{align}\label{UniInteg5}
\frac{d}{dt}\int_{0}^{\infty}  \Psi_{2, \lambda}(h_{\varepsilon}(\mu, t)) d\mu  \le &\frac{1}{\varepsilon}\int_{0}^{\infty}\int_{\nu}^{\infty}
\{ \Phi_{2, \lambda}(h_{\varepsilon}(\mu +\varepsilon \nu, t))-\Phi_{2, \lambda}(h_{\varepsilon}(\mu, t)) \}\Lambda(\mu, \nu)  \zeta_{\varepsilon}(\nu, t) d\mu d\nu\nonumber \\
= &\frac{1}{\varepsilon}\int_{0}^{\infty}\int_{(1+\varepsilon) \nu}^{\infty} \Phi_{2,\lambda}(h_{\varepsilon}(\mu, t))\Lambda(\mu-\varepsilon \nu, \nu) \zeta_{\varepsilon}(\nu, t)d\mu d\nu\nonumber \\
&-\frac{1}{\varepsilon}\int_{0}^{\infty}\int_{\nu}^{\infty}\Phi_{2, \lambda}(h_{\varepsilon}(\mu, t)) \Lambda(\mu, \nu) \zeta_{\varepsilon}(\nu, t) d\mu d\nu \nonumber\\
\le &\frac{1}{\varepsilon}\int_{0}^{\infty}\int_{ \nu}^{\infty} \{ \Lambda(\mu-\varepsilon \nu, \nu)-\Lambda(\mu, \nu) \} \Phi_{2,\lambda}(h_{\varepsilon}(\mu, t)) \zeta_{\varepsilon}(\nu, t)d\mu d\nu \nonumber \\
=&\frac{1}{\varepsilon}\int_{0}^{\infty}\int_{ 0}^{\mu} \{ \Lambda(\mu-\varepsilon \nu, \nu)-\Lambda(\mu, \nu) \} \Phi_{2,\lambda}(h_{\varepsilon}(\mu, t)) \zeta_{\varepsilon}(\nu, t)d\nu d\mu.
\end{align}
Using \eqref{bound coagulation kernel} and Lemma \ref{LemmaUniformbound1} into \eqref{UniInteg5}, we have
\begin{align}\label{UniInteg51}
\frac{d}{dt}\int_{0}^{\infty}  \Psi_{2, \lambda}(h_{\varepsilon}(\mu, t)) d\mu  \le
&-\int_{0}^{\infty}\int_{ 0}^{\mu} \nu \Lambda_{\mu}(\mu, \nu)  \Phi_{2,\lambda}(h_{\varepsilon}(\mu, t)) \zeta_{\varepsilon}(\nu, t)d\nu d\mu\nonumber\\
\le
&\eta \int_{0}^{\infty}\int_{ 0}^{\mu} \nu \mu^{-\sigma-1} \nu^{-\sigma}  \Phi_{2,\lambda}(h_{\varepsilon}(\mu, t)) \zeta_{\varepsilon}(\nu, t)d\nu d\mu\nonumber\\
\le
&\eta \int_{0}^{\infty} \int_{ 0}^{\mu}  \nu^{-2\sigma}  \Phi_{2,\lambda}(h_{\varepsilon}(\mu, t)) \zeta_{\varepsilon}(\nu, t)d\nu d\mu \nonumber\\
 \le & \eta \Theta \int_{0}^{\infty}  \Phi_{2,\lambda}(h_{\varepsilon}(\mu, t)) d\mu \le \eta \Theta \int_{0}^{\infty}  \Psi_{2,\lambda}(h_{\varepsilon}(\mu, t)) d\mu.
\end{align}

Thanks to \eqref{convexp2} and Gronwall's inequality, \eqref{UniInteg51} leads to
 \begin{align}\label{UniInteg6}
 \int_{0}^{\infty} \Psi_{2, \lambda}(h_{\varepsilon}(\mu, t)) d\mu \le \Theta_1(T).
\end{align}
Finally, passing the limit $\lambda \to \infty$ to \eqref{UniInteg6} which gives \eqref{UniformInt2}. This completes the proof of Lemma \ref{LemmaUniformInt1}.
\end{proof}


\subsection{Existence and uniqueness of solutions for truncated problem}
In this subsection, we discuss the existence and uniqueness of non-negative solutions for truncated problems.
Fix $\varepsilon \in(0,1]$ and $ n \in \mathbb{N}$. Let us truncate the coagulation rate $\Lambda$ by  $\Lambda_n$ as

\begin{equation}\label{trunc Coag}
\Lambda_{n}(\mu, \nu):=\begin{cases}
\Lambda(\mu, \nu),\ \ & \text{if}\ (\mu, \nu)\in(1/n, n)^{2}, \\
\text{0},\ \ &  \text{otherwise}.
\end{cases}
\end{equation}

Inserting \eqref{trunc Coag} into \eqref{GenerSCEeps}, the truncated generalized Smoluchowski coagulation equation can be expressed as
\begin{align}\label{RegluGener}
\partial_{t}(\zeta_{\varepsilon, n})=Q_{\varepsilon, n}(\zeta_{\varepsilon, n})\ \ \ (\mu, t) \in (0,n) \times (0, \infty)
\end{align}
and
\begin{equation}\label{RegluGenerin}
\zeta_{\varepsilon, n}(0)= \zeta^{in}_{\varepsilon, n}:=\zeta^{in} \chi_{(0,n)}, \ \ \ \mu \in (0,n),
\end{equation}
where
\begin{align*}
Q_{\varepsilon, n}(\zeta_{\varepsilon, n})=Q^1_{\varepsilon, n}(\zeta_{\varepsilon, n})
-Q^2_{\varepsilon, n}(\zeta_{\varepsilon, n}),
\end{align*}
\begin{align*}
Q^1_{\varepsilon, n}(\zeta_{\varepsilon, n}) (\mu, t)=&\frac{1}{\varepsilon}\int_{0}^{\mu/(1+\varepsilon)} \Lambda_n(\mu-\varepsilon \nu, \nu)\zeta_{\varepsilon, n}(\mu-\varepsilon \nu, t)\zeta_{\varepsilon, n}(\nu, t)d\nu,
\end{align*}
\begin{align*}
Q^2_{\varepsilon, n}(\zeta_{\varepsilon, n})(\mu, t) =& \zeta_{\varepsilon, n}(\mu, t)
\bigg[\bigg(\frac{1}{\varepsilon}-1\bigg)\int_{0}^{\mu} \Lambda_{n}(\mu, \nu) \zeta_{\varepsilon, n}(\nu, t)d\nu
+\int_{0}^{n}\Lambda_{n}(\mu, \nu) \zeta_{\varepsilon, n}(\nu, t)d\nu \bigg].
\end{align*}

Next, we state the weak formulation to the truncated generalized equations \eqref{RegluGener}--\eqref{RegluGenerin} as
\begin{align}\label{Weaktruncsolunger}
\int_0^n &\omega(\mu)\zeta_{\varepsilon, n}(\mu, t) d\mu- \int_0^n \omega(\mu) \zeta^{in}_{\varepsilon, n}(\mu) d\mu \nonumber\\
= & \int_0^t \int_0^n \int_{\nu}^{n}  \frac{ \{ \omega(\mu+\varepsilon \nu)\chi_{(0, n)}(\mu+\varepsilon \nu)-  \omega(\mu) \} } {\varepsilon} \Lambda_n(\mu, \nu) \zeta_{\varepsilon, n}(\mu, s) \zeta_{\varepsilon, n}(\nu, s) d\mu d\nu ds \nonumber\\
&- \int_0^t \int_0^n \int_0^{\nu} \omega(\mu) \Lambda_n(\mu, \nu) \zeta_{\varepsilon, n}(\mu, s) \zeta_{\varepsilon, n}(\nu, s) d\mu d\nu ds,
\end{align}
for $\omega \in L^{\infty}(0, n)$.\\

Let us state the following theorem which is useful to prove coming Proposition \ref{proppooja}.
\begin{thm}\label{Initialvaluethm}
Let $B$ be a Banach space. If $G: B \to B$ is a locally Lipschitz function in $B$ and if $\zeta^{in} \in B$, then there exists a unique maximal solution $\zeta \in \mathcal{C}^1([0, T_1); B) $ to the initial value problem
\begin{align*}
\frac{d \zeta}{dt}= G(\zeta),
\end{align*}
and
\begin{align*}
\zeta(0)= \zeta^{in},
\end{align*}
and either
\begin{align*}
\text{Case-I:}\ \ T_1=\infty,\ \ \text{or} \ \ \text{Case-II:}\ \  T_1 < \infty\  \text{and}\ \lim_{t \to T_1}\|\zeta(t)\|=\infty.
\end{align*}
\end{thm}


\begin{proposition}\label{proppooja}
Let $\varepsilon \in (0, 1]$. Then, for each $n \ge 1$, there exists a unique non-negative solution $\zeta_{\varepsilon, n} \in \mathcal{C}^1([0,\infty);L^1((0, n); d\mu))$ to \eqref{RegluGener}--\eqref{RegluGenerin} which satisfies
\begin{align}\label{Barikprop}
\int_0^{n}  \mu \zeta_{\varepsilon, n} (\mu, t)d\mu  \le   \int_0^{n} \mu \zeta^{in}_{\varepsilon, n}(\mu) d\mu,
\end{align}
and
\begin{align}\label{Bariknorn}
\| \zeta_{\varepsilon, n} (t) \|_{L^1((0, n); d\mu)}  \le   \| \zeta^{in}_{\varepsilon, n} \|_{L^1((0, n); d\mu)},
\end{align}
for $t\ge 0$.
\end{proposition}

\begin{proof}
From \eqref{coagulation kernel} and \eqref{trunc Coag}, we can see that
\begin{align}\label{Estimate_Coa}
\Lambda_{n}(\mu, \nu)\le 2 k n^{2+2\sigma},\ \ \ \ \forall (\mu, \nu)\in (1/n, n)^{2}.
\end{align}

In order to prove the existence of a unique non-negative solution to \eqref{RegluGener}--\eqref{RegluGenerin}, it is sufficient to apply Theorem \ref{Initialvaluethm}. For this, we first need to show that $Q^{i}_{\varepsilon, n},\ i=1,2$ are locally Lipschitz functions from $L^{1}((0, n); d\mu)$ into $L^{1}((0, n); d\mu)$. Let $\zeta^{1}$ and $\zeta^{2}$ be two solutions in $L^{1}((0, n); d\mu)$, then we have
\begin{align}\label{Inithm1}
\int_{0}^{n} & |Q^{1}_{\varepsilon, n}(\zeta^{1}_{\varepsilon, n})(\mu, t)-Q^{1}_{\varepsilon, n}(\zeta^{2}_{\varepsilon, n})(\mu, t)|d\mu \nonumber\\
 \le  & \frac{1}{\varepsilon}\int_{0}^{n}\int_{0}^{\mu/(1+\varepsilon)}|\Lambda_{n}(\mu-\varepsilon \nu, \nu) \{ \zeta^{1}_{\varepsilon, n}(\mu-\varepsilon \nu, t) \zeta^{1}_{\varepsilon, n}(\nu, t)- \zeta^{2}_{\varepsilon, n}(\mu-\varepsilon \nu, t) \zeta^{2}_{\varepsilon, n}(\nu, t) \}|d\nu d\mu.
\end{align}

On the one hand, by applying the Fubini theorem, the transformation $\mu-\varepsilon \nu=\mu'$ and $\nu=\nu'$ into \eqref{Inithm1} and then \eqref{Estimate_Coa}, we obtain
\begin{align}\label{Inithm3}
\int_{0}^{n} & |Q^{1}_{\varepsilon, n}(\zeta^{1}_{\varepsilon, n}) (\mu, t)-Q^{1}_{\varepsilon, n}(\zeta^{2}_{\varepsilon, n}) (\mu, t)|d\mu\nonumber\\
= &\frac{1}{\varepsilon}\int_{0}^{n}\int_{\nu'}^{n-\epsilon \nu'}|\Lambda_{n}(\mu^{'}, \nu')||\zeta^{1}_{\varepsilon,n}(\mu', t) \zeta^{1}_{\varepsilon, n}(\nu', t)- \zeta^{2}_{\varepsilon, n}(\mu{'}, t) \zeta^{2}_{\varepsilon, n}(\nu', t)|d\mu'd\nu'\nonumber\\
\le & \frac{2k n^{2+2\sigma}}{\varepsilon}\int_{0}^{n}\int_{0}^{n} \{|\zeta^{1}_{\varepsilon, n}(\mu, t)| |\zeta^{1}_{\varepsilon, n}(\nu, t)-\zeta^{2}_{\varepsilon, n}(\nu, t)|+ |\zeta^{2}_{\varepsilon, n}(\nu, t)| | \zeta^{1}_{\varepsilon, n}(\mu, t)-\zeta^{2}_{\varepsilon, n}(\mu, t)| \}d\mu d\nu \nonumber\\
= &\frac{2k n^{2+2\sigma}}{\varepsilon} \{ \|\zeta^{1}_{\varepsilon, n}\|_{L^{1}((0, n); d\mu)}
+\|\zeta^{2}_{\varepsilon, n}\|_{L^{1}((0, n); d\mu)} \} \|\zeta^{1}_{\varepsilon, n}-\zeta^{2}_{\varepsilon, n}\|_{L^{1}((0, n); d\mu)}.
\end{align}
On the other hand, again \eqref{Estimate_Coa} helps to show that $Q^{2}_{\varepsilon, n}$ is locally Lipschitz function, i.e.,
\begin{align}\label{Inithm4}
\int_{0}^{n}|Q^{2}_{\varepsilon, n}&(\zeta^{1}_{\varepsilon, n})(\mu, t)-Q^{2}_{\varepsilon, n}(\zeta^{2}_{\varepsilon, n})(\mu, t)|d\mu\nonumber\\
=& \bigg(\frac{1}{\varepsilon}-1\bigg)\int_{0}^{n} \int_{0}^{\mu}
|\Lambda_{n}(\mu, \nu)|
|\zeta^{1}_{\varepsilon, n}(\mu, t) \zeta^{1}_{\varepsilon, n}(\nu, t)-\zeta^{2}_{\varepsilon, n}(\mu, t)
\zeta^{2}_{\varepsilon, n}(\nu, t)|d\nu d\mu \nonumber\\
&+\int_{0}^{n}\int_{0}^{n}|\Lambda_{n}(\mu, \nu)|
|\zeta^{1}_{\varepsilon,n}(\mu, t) \zeta^{1}_{\varepsilon, n}(\nu, t)-\zeta^{2}_{\varepsilon, n}(\mu, t) \zeta^{2}_{\varepsilon, n}(\nu, t)|d\nu d\mu\nonumber\\
\le & 2kn^{2+2\sigma}\bigg(\frac{1}{\varepsilon}+2\bigg)
\{ \|\zeta^{1}_{\varepsilon, n}\|_{L^{1}((0, n); d\mu)}+\|\zeta^{2}_{\varepsilon, n}\|_{L^{1}((0, n); d\mu)} \} \|\zeta^{1}_{\varepsilon, n}
-\zeta^{2}_{\varepsilon, n}\|_{L^{1}((0, n); d\mu)}.
\end{align}

From \eqref{Inithm3} and \eqref{Inithm4}, it is clear that $Q^{i}_{\varepsilon, n}, i=1,2$ are  locally Lipschitz functions from $L^{1}((0, n); d\mu)$ to $L^{1}((0, n); d\mu)$. Then, by Theorem \ref{Initialvaluethm}, there exists a unique solution $\zeta_{\varepsilon, n} \in \mathcal{C}^{1}([0,T_{\ast});L^{1} ((0, n); d\mu))$ to \eqref{RegluGener}--\eqref{RegluGenerin}. Next, we want to show that $\zeta_{\varepsilon, n}$ is non-negative. Let us introduce the positive part of a real number $p$ as $[p]_{+}=\max\{p, 0\}$. Since $Q^{1}_{\varepsilon, n}$ is  locally Lipschitz, thus $[Q^{1}_{\varepsilon, n}]_{+}$ is also a  locally Lipschitz function. Hence, the initial value problem
\begin{align}\label{RegluGener1}
\partial_{t}(\zeta_{\varepsilon, n})=[Q^1_{\varepsilon, n}(\zeta_{\varepsilon, n})]_{+} - Q^2_{\varepsilon, n}(\zeta_{\varepsilon, n})
\end{align}
and
\begin{equation}\label{RegluGenerin1}
\zeta_{\varepsilon, n}(0)= \zeta^{in}_{\varepsilon, n},
\end{equation}
has a unique solution by Theorem \ref{Initialvaluethm}. Next, to prove the non-negativity of $\zeta_{\varepsilon,n}$,  we apply the chain rule
 \begin{align}\label{RegluGener2}
 \frac{\partial}{\partial t}(-\zeta_{\varepsilon,n})_{+}= -\mathrm{sign}_{+}(-\zeta_{\varepsilon,n})\frac{\partial}{\partial t}(\zeta_{\varepsilon,n}),
 \end{align}
 where $\mathrm{sign}_{+}(x)=1$ for $x\geq 0$ and  $\mathrm{sign}_{+}(x)=0$ for $x< 0$.
We infer from \eqref{RegluGener1},\eqref{Estimate_Coa} and \eqref{RegluGener2} that

\begin{align*}
 \frac{d}{dt} \|(-{\zeta_{\varepsilon, n}})_{+} (t) & \|_{L^1((0, n); d\mu)}\\
&=-\int_0^n \mathrm{sign}_+(-\zeta_{\varepsilon, n})(\mu, t) \left[ {Q}^1_{\varepsilon, n}(\zeta_{\varepsilon, n})(\mu, t) \right]_+ d\mu \\
& \qquad\qquad + \int_0^n \mathrm{sign}_+(-\zeta_{\varepsilon, n})(\mu, t) {Q}^2_{\varepsilon, n}(\zeta_{\varepsilon, n})(\mu, t) d\mu \\
& \qquad \le - \int_0^n \mathrm{sign}_+(-\zeta_{\varepsilon, n})(\mu, t)(- \zeta_{\varepsilon, n} (\mu, t))
\bigg[\bigg(\frac{1}{\varepsilon}-1\bigg)\int_{0}^{\mu} \Lambda_{n}(\mu, \nu) \zeta_{\varepsilon, n}(\nu, t)d\nu \nonumber\\
&  \qquad\qquad +\int_{0}^{n}\Lambda_{n}(\mu, \nu) \zeta_{\varepsilon, n}(\nu, t)d\nu \bigg] d\mu \\
& \qquad \le 2k n^{2+2\sigma}  \int_0^n (-\zeta_{\varepsilon, n})_+(t)
\bigg[\bigg(\frac{1}{\varepsilon}+1\bigg)\int_{0}^{\mu}  \zeta_{\varepsilon, n}(\nu, t)d\nu +\int_{0}^{n} \zeta_{\varepsilon, n}(\nu, t)d\nu \bigg] d\mu \\
& \qquad \le 2k n^{2+2\sigma} \bigg(\frac{1}{\varepsilon}+2\bigg) \|\zeta_{\varepsilon, n}(t)\|_{L^1((0, n); d\mu)} \| (-\zeta_{\varepsilon, n})_+(t)\|_{L^1((0, n); d\mu)}.
\end{align*}

Using an application of Gronwall's lemma, we end up with
\begin{align*}
\|(-{\zeta_{\varepsilon, n}}&(t))_{+}\|_{L^1((0, n); d\mu)} \nonumber\\
\le  & \|(-{\zeta^{in}_{\varepsilon, n})}_{+}\|_{L^1((0, n); d\mu)} \exp{\left( 2k n^{2+2\sigma} \bigg(\frac{1}{\varepsilon}+2\bigg) \int_0^t \|\zeta_{\varepsilon, n}(s)\|_{L^1((0, n); d\mu)} ds \right)}
\end{align*}
for $t\in [0, T_{\ast})$.
Next, the non-negativity of $\zeta^{in}_{\varepsilon, n}$ from \eqref{RegluGenerin} confirms that
\begin{equation*}
\|(-\zeta^{in}_{\varepsilon, n})_{+}\|_{L^1((0, n); d\mu)}=0.
\end{equation*}
Thus, we have $\zeta_{\varepsilon, n}(t) \ge 0$ a.e.\ in $\mathbb{R}_{>0}$.\\ \\
 Next, in order to prove \eqref{Bariknorn}, let us consider the following from \eqref{Weaktruncsolunger} by setting $\omega \equiv 1$
\begin{align*}
 \| \zeta_{\varepsilon, n} (t)\|_{L^1((0, n); d\mu)} = &\| \zeta^{in}_{\varepsilon, n} \|_{L^1((0, n); d\mu)}
- \int_0^t \int_0^n \int_0^{\nu} \Lambda_n(\mu, \nu) \zeta_{\varepsilon, n}(\mu, s) \zeta_{\varepsilon, n}(\nu, s) d\mu d\nu ds \nonumber\\
&-\frac{1}{\varepsilon} \int_0^t \int_0^{n} \int_{n-\varepsilon n}^{\nu} \Lambda_n(\mu, \nu) \zeta_{\varepsilon, n}(\mu, s). \zeta_{\varepsilon, n}(\nu, s) d\mu d\nu ds \nonumber\\
\end{align*}
Since $\zeta_{\varepsilon, n}(t) \ge 0$ a.e.\ in $\mathbb{R}_{>0}$ and $ \Lambda_n \ge 0$, hence, the above equation implies that
\begin{align}\label{Weaktruncsolunger3}
 \| \zeta_{\varepsilon, n} (t) \|_{L^1((0, n); d\mu)} \le  \| \zeta^{in}_{\varepsilon, n} \|_{L^1((0, n); d\mu)},\ \ \text{for}\ t \in [0, T_{\ast}),
\end{align}
which prevents the case-II (given in Theorem \ref{Initialvaluethm} ). Hence, $T_{\ast} =\infty$. Analogously, \ref{Barikprop} can be obtained by setting $\omega(\mu) = \mu \chi_{(0, n)}(\mu)$ into \eqref{Weaktruncsolunger}.
\end{proof}

Next, it can easily be shown that $\zeta_{\varepsilon, n}$ satisfies similar estimates in Lemmas \ref{LemmaUniformbound1}--\ref{LemmaUniformInt1} for each $T>0$.
Hence, we have
\begin{align}\label{Uniformboundtruncated}
 \int_{0}^{\infty} ( \mu^{-2\sigma}+\mu) \zeta_{\varepsilon, n}(\mu, t)  d\mu \le \Theta,
\end{align}
\begin{align}\label{LemmalargevalueA}
\sup_{t \in [0, T) }  \int_0^{\infty} \Psi_1(\mu) \zeta_{\epsilon, n} (\mu, t) d\mu < \Theta_2(T),
\end{align}
and
\begin{align}\label{UniformIntegra11}
\sup_{t \in  [0, T)}\int_{0}^{\infty} \Psi_2( \mu^{-\sigma} \zeta_{\varepsilon, n}(\mu, t)) d\mu \le \Theta_{1}(T),
\end{align}
uniformly with respect to $\varepsilon \in (0, 1]$ and $n \in \mathbb{N}$, and
where $\Psi_1$ and $\Psi_2$ are two convex functions whose first derivatives are concave and satisfy \eqref{convexp1}--\eqref{convexp2} and Lemma \ref{convexlambda}.  \\
From \eqref{UniformIntegra11} and the Dunford-Pettis theorem, we infer that there exists a weakly relatively compact subset $\mathcal{W}_T$ of $L^1( \mathbb{R}_{>0};  \mu^{-\sigma} d\mu )$ with respect to space variable $\mu$.

\subsection{Time Equicontinuity}
In this section, we turn to check the equicontinuity with respect to  time variable $t$ in the weak topology of $L^1( \mathbb{R}_{>0};  \mu^{-\sigma} d\mu )$. Let us first consider $\omega\in \mathcal{D}(\mathbb{R}_{>0})$ and $\lambda \in (1, n)$.  For $\eta \in (0, 1)$ and $0 \le t_1 < t_2 \le T$, set
$\omega_{3}(\mu) \equiv \omega(\mu) (\mu+\eta)^{-\sigma}$ into \eqref{Weaktruncsolunger}, where $\omega(\mu)=\omega(\mu) \chi_{(0, \lambda)}(\mu)$, we have
\begin{align}\label{Equicon1}
\bigg| \int_0^{\lambda} & \omega_{3}(\mu)  \{ \zeta_{\varepsilon, n} (\mu, t_2) - \zeta_{\varepsilon, n} (\mu, t_1) \} d\mu   \bigg| \nonumber\\
\le &  \frac{1}{\varepsilon}  \int_{t_1}^{t_2}   \int_0^n \int_{\nu}^n (\mu+\eta)^{-\sigma} | \omega(\mu+\varepsilon \nu)\chi_{(0, \lambda)}(\mu+\varepsilon \nu) -
\omega(\mu)\chi_{(0, \lambda)}(\mu) | \nonumber\\
 &~~~~~~~~~~ \times \Lambda_n(\mu, \nu) \zeta_{\varepsilon, n}(\mu, s) \zeta_{\varepsilon, n}(\nu, s) d\mu d\nu ds \nonumber\\
&+ \int_{t_1}^{t_2} \int_0^n \int_0^{\nu} (\mu+\eta)^{-\sigma} |\omega(\mu) |\chi_{(0, \lambda)}(\mu) \Lambda_n(\mu, \nu) \zeta_{\varepsilon, n}(\mu, s) \zeta_{\varepsilon, n}(\nu, s) d\mu d\nu ds  \nonumber\\
 = &   \int_{t_1}^{t_2}   \int_0^{\lambda} \int_{\lambda-\varepsilon \nu}^{\lambda} (\mu+\eta)^{-\sigma} \frac{ | \omega(\mu+\varepsilon \nu)- \omega(\mu) | } {\varepsilon} \Lambda_n(\mu, \nu) \zeta_{\varepsilon, n}(\mu, s) \zeta_{\varepsilon, n}(\nu, s) d\mu d\nu ds \nonumber\\
&+ \int_{t_1}^{t_2} \int_0^{n} \int_0^{\nu} (\mu+\eta)^{-\sigma} |\omega(\mu)| \Lambda_n(\mu, \nu) \zeta_{\varepsilon, n}(\mu, s) \zeta_{\varepsilon, n}(\nu, s) d\mu d\nu ds.
\end{align}

One can infer from \eqref{coagulation kernel}, Lemma \ref{LemmaUniformbound1} and \eqref{Equicon1} that
\begin{align*}
\bigg| \int_0^{\lambda}  (\mu+\eta)^{-\sigma}\omega(\mu) & \{ \zeta_{\varepsilon, n} (\mu, t_2) - \zeta_{\varepsilon, n} (\mu, t_1) \} d\mu   \bigg| \nonumber\\
\le & \|\omega\|_{W^{1, \infty}(\mathbb{R}_{>0})} \int_{t_1}^{t_2}   \int_0^{\lambda} \int_{\lambda-\varepsilon \nu}^{\lambda} (\mu+\eta)^{-\sigma} \nu \Lambda_n(\mu, \nu) \zeta_{\varepsilon, n}(\mu, s) \zeta_{\varepsilon, n}(\nu, s) d\mu d\nu ds \nonumber\\
&+ \|\omega \|_{L^{\infty}(\mathbb{R}_{>0})} \int_{t_1}^{t_2} \int_0^{n} \int_0^{\nu} \mu^{-\sigma}  \Lambda_n(\mu, \nu) \zeta_{\varepsilon, n}(\mu, s) \zeta_{\varepsilon, n}(\nu, s) d\mu d\nu ds \nonumber\\
  \le & 4 k(\|\omega\|_{W^{1, \infty}(\mathbb{R}_{>0})}  + \|\omega \|_{L^{\infty}(\mathbb{R}_{>0})} ) \Theta^2 (t_2 -t_1).
\end{align*}
As $\eta \to 0$, then by Fatou's lemma, we obtain
 \begin{align}\label{Equicon2}
\bigg| \int_0^{\lambda} \mu^{-\sigma} \omega(\mu) \{ \zeta_{\varepsilon, n} (\mu, t_2) - \zeta_{\varepsilon, n} (\mu, t_1) \} d\mu   \bigg|
  \le  k (\|\omega\|_{W^{1, \infty}(\mathbb{R}_{>0})}  + \|\omega \|_{L^{\infty}(\mathbb{R}_{>0})} ) \Theta^2 (t_2 -t_1).
\end{align}
Next, by using Lemmas \ref{LemmaUniformbound1}--\ref{Lemmalargevalue} and \eqref{Equicon2}, we estimate the following term as
 \begin{align*}
\bigg| \int_0^{\infty} &  \mu^{-\sigma} \omega(\mu) \{ \zeta_{\varepsilon, n} (\mu, t_2) - \zeta_{\varepsilon, n} (\mu, t_1) \} d\mu   \bigg|
\le  \bigg| \int_0^{\lambda} \mu^{-\sigma} \omega(\mu) \{ \zeta_{\varepsilon, n} (\mu, t_2) - \zeta_{\varepsilon, n} (\mu, t_1) \} d\mu   \bigg| \nonumber\\
&+\bigg| \int_{\lambda}^{\infty} \mu^{-\sigma} \omega(\mu) \{ \zeta_{\varepsilon, n} (\mu, t_2) - \zeta_{\varepsilon, n} (\mu, t_1) \} d\mu   \bigg| \nonumber\\
  \le & k (\|\omega\|_{W^{1, \infty}(\mathbb{R}_{>0})}  + \|\omega \|_{L^{\infty}(\mathbb{R}_{>0})} ) \Theta^2 (t_2 -t_1) + \|\omega \|_{L^{\infty}} \Theta_2(T) \frac{\lambda} {\Psi_1(\lambda)}.
\end{align*}

As $\lambda \to \infty$, then by \eqref{convexp1}, we have
 \begin{align*}
\bigg| \int_0^{\infty}   \mu^{-\sigma} \omega(\mu) \{ \zeta_{\varepsilon, n} (\mu, t_2) - \zeta_{\varepsilon, n} (\mu, t_1) \} d\mu   \bigg|
  \le  k (\|\omega\|_{W^{1, \infty}(\mathbb{R}_{>0})}  + \|\omega \|_{L^{\infty}}(\mathbb{R}_{>0}) ) \Theta^2 (t_2 -t_1).
\end{align*}
This proves the time equicontinuity in $L^1( \mathbb{R}_{>0};  \mu^{-\sigma} d\mu )$ when $\omega \in \mathcal{D}(\mathbb{R}_{>0})$.
\\

Next, we will check the time equi-continuity of family of solutions $ \{ \zeta_{\varepsilon, n}\}_{n >1} \subset L^1( \mathbb{R}_{>0};  \mu^{-\sigma} d\mu )$ when $\omega \in L^{\infty} (\mathbb{R}_{>0})$.\\

Let us assume $\Phi_{\varepsilon, n}(\mu, t)=\mu^{-\sigma}\zeta_{\varepsilon, n}(\mu,t)$ for every $0\le t_{1} < t_{2}\le T$. Then for $\omega \in L^{\infty}(\mathbb{R}_{>0})$, there exists a family of functions $\omega_{k} \in \mathcal{D}(\mathbb{R}_{>0})$ such that
\begin{align}\label{EquiContinuous1}
\omega_{k}(\mu)\rightarrow \omega(\mu)~~~\mbox{a.e.}~~\mbox{in}~~ \mathbb{R}_{>0}~\mbox{with}~\|\omega_{k}\|_{L^{\infty}(\mathbb{R}_{>0})}\le c\|\omega\|_{L^{\infty}(\mathbb{R}_{>0})},
\end{align}
where $c>0$ is a constant and supp$(\omega_{k})\subset[0, \lambda_0] \subset [0, n), \lambda_{0}>1$.\\

Let us simplify the following integral, by using the triangle inequality, as
 \begin{align}\label{EquiContinuous2}
\bigg|\int_{0}^{\infty} \{  \Phi_{\varepsilon, n}(\mu,t_{2})-\Phi_{\varepsilon, n}(\mu,t_{1}) \} \omega(\mu) d\mu\bigg|
\le &\bigg|\int_{0}^{\infty} \{ \Phi_{\varepsilon, n}(\mu,t_{2})-\Phi_{\varepsilon, n}(\mu,t_{1}) \}  \omega_{k}(\mu) d\mu\bigg|\nonumber\\ &+\bigg|\int_{0}^{\infty}  \{  \Phi_{\varepsilon, n}(\mu,t_{2})-\Phi_{\varepsilon, n}(\mu,t_{1}) \}
(\omega-\omega_{k})(\mu) d\mu\bigg|\nonumber\\
=:&\sum_{i=1}^{2} \mathcal{I}_{i}.
\end{align}
where $\mathcal{I}_{i}$, for $i=1,2$, are the first and second integral on the right-hand side of \eqref{EquiContinuous2}.
Next, one can infer from the Egorov theorem that
\begin{align}\label{EquiContinuous3}
\lim_{k \rightarrow \infty}\sup_{\mu \in (0, \lambda_0) \setminus \mathcal{E}_{\delta, \lambda_0}}|\omega_{k}(\mu)-\omega(\mu)|=0,~~~\mbox{for}~\mbox{every}~ \lambda_0>1~\mbox{and}~\delta\in(0, 1),
\end{align}
where $\mathcal{E}_{\delta, \lambda_0}$ is a measurable subset of $(0, \lambda_0)$ with Lebesgue measure $|\mathcal{E}_{\delta, \lambda_0}| \le \delta$.\\

From \eqref{EquiContinuous1}, we evaluate the estimate value of $\mathcal{I}_{1}$ as
\begin{align}\label{EquiContinuous4}
 \mathcal{I}_1 \le  k c (\|\omega\|_{W^{1, \infty}(\mathbb{R}_{>0})}  + \|\omega \|_{L^{\infty}(\mathbb{R}_{>0}) }  \Theta^2 |t_{2}-t_{1}|.
\end{align}

Next, we evaluate $\mathcal{I}_{2}$, by applying Lemma \ref{LemmaUniformbound1}, as
\begin{align}\label{EquiContinuous5}
\mathcal{I}_{2} = & \bigg|\int_{0}^{\infty} \{ \Phi_{\varepsilon, n}(\mu,t_{2})-\Phi_{\varepsilon, n}(\mu, t_{1}) \} (\omega-\omega_{k})(\mu) d\mu\bigg|\nonumber\\
\le & \int_{(0,\lambda_0) \setminus \mathcal{E}_{\delta, \lambda_0}}| \{ \Phi_{\varepsilon, n}(\mu,t_{2})-\Phi_{\varepsilon, n}(\mu,t_{1}) \} (\omega-\omega_{k})(\mu)| d\mu \nonumber\\
&+\int_{\mathcal{E}_{\delta, \lambda_0}}| \{ \Phi_{\varepsilon, n}(\mu,t_{2})-\Phi_{\varepsilon, n}(\mu,t_{1}) \}
(\omega-\omega_{k})(\mu)| d\mu\nonumber\\
&+\int_{\lambda_0}^{\infty}| \{ \Phi_{\varepsilon, n}(\mu,t_{2})-\Phi_{\varepsilon, n}(\mu,t_{1}) \} (\omega-\omega_{k})(\mu)| d\mu.\nonumber\\
\le & 2\Theta \sup_{\mu \in((0,\lambda_0) \setminus \mathcal{ E}_{\delta, \lambda_0})} |\omega_{k}(\mu)-\omega(\mu)| +\int_{\mathcal{E}_{\delta, \lambda_0}}| \{ \Phi_{\varepsilon, n}(\mu,t_{2})-\Phi_{\varepsilon, n}(\mu,t_{1}) \}
(\omega-\omega_{k})| d\mu\nonumber\\
&+\int_{\lambda_0}^{\infty}| \{ \Phi_{\varepsilon, n}(\mu,t_{2})-\Phi_{\varepsilon, n}(\mu,t_{1}) \}(\omega-\omega_{k})(\mu)| d\mu.
\end{align}
As $k \to \infty$ and using \eqref{EquiContinuous1} and \eqref{EquiContinuous3} into \eqref{EquiContinuous5}, we get
 \begin{align}\label{EquiContinuous6}
\mathcal{I}_{2}
\le &(c+1)\|\omega\|_{L^{\infty}(\mathbb{R}_{>0})} \bigg\{ \int_{\lambda_0}^{\infty} +\int_{\mathcal{E}_{\delta, \lambda_0}}  \bigg\}  \{ \mu^{-\sigma} \zeta_{\varepsilon, n}(\mu,t_{2})+\mu^{-\sigma} \zeta_{\varepsilon, n}(\mu,t_{1}) \} d\mu.
\end{align}
From \eqref{Uniformboundtruncated}, we have
\begin{align*}
\int_{\lambda_0}^{\infty} \mu^{-\sigma} \zeta_{\varepsilon, n}(\mu,t)d\mu \le \Theta~\mbox{and}~\int_{\mathcal{E}_{\delta, \lambda_0}}  \mu^{-\sigma} \zeta_{\varepsilon, n}(\mu,t)d\mu \le \Theta
\end{align*}
for $t \in (0, \infty)$. Using Lebesgue's dominated convergence theorem, we obtain
\begin{align}\label{EquiContinuous7}
\int_{\lambda_0}^{\infty} \mu^{-\sigma} \zeta_{\varepsilon, n}(\mu,t)d\mu \rightarrow 0, \ \text{as}\ \lambda_0 \rightarrow \infty.
\end{align}
Similarly, as $\lambda_0 \rightarrow \infty$ and $\delta \rightarrow 0$, we have
\begin{align}\label{EquiContinuous8}
\int_{\mathcal{E}_{\delta, \lambda_0}} \mu^{-\sigma} \zeta_{\varepsilon, n}(\mu, t)d\mu \rightarrow 0.
\end{align}

Now, taking the limit $\lambda_{0}\rightarrow \infty$ and $\delta\rightarrow0$ into \eqref{EquiContinuous6} and then using \eqref{EquiContinuous7} and \eqref{EquiContinuous8}, we have
\begin{align}\label{EquiContinuous 4}
\mathcal{I}_{2} \rightarrow 0.
\end{align}
 From  \eqref{EquiContinuous4} and \eqref{EquiContinuous 4} into \eqref{EquiContinuous2}, we obtain
\begin{align}\label{TimeEquicontin}
\bigg|\int_{0}^{\infty} \{ \Phi_{\varepsilon, n}(\mu, t_{2})-\Phi_{\varepsilon, n}(\mu, t_{1}) \} \omega(\mu) d\mu\bigg|\le C_{1}(\omega)|t_{2}-t_{1}|,
 \end{align}
  for $|t_{2}-t_{1}|<\delta_{1}$ for some sufficiently small $\delta_{1}$, this proves the equi-continuity result. One can easily see that it is also true for $t_{1}> t_{2}$. Thus, the family of functions $(\zeta_{\varepsilon, n})$ is a time equi-continuous in the topology $L^{1}( \mathbb{R}_{>0}; \mu^{-\sigma} d\mu)$.\\

   Then according to a refined version of Arzel\`{a}-Ascoli theorem \cite[Definition 1.3.1.]{Vrabie:1995} there exists a subsequence (not relabeled) $\zeta_{\varepsilon,n}$ and $\zeta_{\varepsilon} \in \mathcal{C}_w[(0, T); L^1(\mathbb{R}_{>0}, \mu^{-\sigma} d\mu ))$ such that
 \begin{align}\label{weakconver1}
  \zeta_{\varepsilon, n}(\mu, t) \rightarrow  \zeta_{\varepsilon} \ \ \text{in} \ \  \mathcal{C}_w([0, T); L^1(\mathbb{R}_{>0}, \mu^{-\sigma} d\mu )).
 \end{align}
 Next, from Lemma \ref{LemmaUniformbound1}, \eqref{LemmalargevalueA} and \eqref{convexp1}, one can easily improve the convergence in \eqref{weakconver1} to
  \begin{align}\label{weakconver2}
  \zeta_{\varepsilon, n}(\mu, t) \rightarrow  \zeta_{\varepsilon} \ \ \text{in} \ \  \mathcal{C}_w([0, T); L^1(\mathbb{R}_{>0}, (\mu^{-\sigma}+\mu) d\mu )).
 \end{align}
 As $\varepsilon \to 0$, we have
 \begin{align}\label{weakconver3}
  \lim_{\varepsilon \to 0}\zeta_{\varepsilon} \rightarrow  \zeta \ \ \text{in} \ \  \mathcal{C}_w([0, T); L^1(\mathbb{R}_{>0}, (\mu^{-\sigma}+\mu) d\mu )).
 \end{align}
 In order to complete the existence result of Theorem \ref{ThemExistence}, we need to check $\zeta$ is indeed a weak solution to \eqref{GenerSCEeps}--\eqref{GenerSCEepsin} which is shown in the next subsection.

\subsection{Weak Convergence of Integral Operators}
  Let $\omega \in\mathcal{D}([0, \infty))$ be a test function  with compact support contained in $(0, \lambda) \subset (0, n)$, for $1< \lambda <n$. From \eqref{WeakGen}, we have
\begin{align}\label{Passlim1}
\int_{0}^{\infty} \zeta_{\varepsilon, n }(\mu, t) \omega(\mu)d\mu
=&\int_{0}^{\infty} \zeta^{in}(\mu) \omega(\mu) d\mu +\mathcal{P}^{1}_{\varepsilon, n}(t)-\mathcal{P}^{2}_{\varepsilon, n}(t)
\end{align}
where
\begin{align*}
\mathcal{P}^{1}_{\varepsilon, n}(t)=&\int_{0}^{t}\int_{0}^{n}\int_{0}^{\nu}\bigg\{\frac{\omega(\nu+\varepsilon \tau)-\omega(\nu)}{\varepsilon}\bigg\}   \Lambda(\nu, \tau) \zeta_{\varepsilon, n}(\nu,s) \zeta_{\varepsilon, n} (\tau, s)  d\tau d\nu ds,\\
\mathcal{P}^{2}_{\varepsilon, n}(t)=& \int_{0}^{t}\int_{0}^{n} \int_{0}^{\nu}
\omega(\tau)  \Lambda(\nu, \tau) \zeta_{\varepsilon, n}(\nu,s)  \zeta_{\varepsilon, n}(\tau, s)  d\tau d\nu ds.
\end{align*}

Since $\zeta_{\varepsilon, n} \rightarrow \zeta_{\varepsilon}$ and $\zeta_{\varepsilon}\rightarrow \zeta $
 in $\mathcal{C}_w([0, T); L^{1}(\mathbb{R}_{>0}; (\mu+\mu^{-\sigma}) d\mu))$, then we have
\begin{align}\label{Passlim2}
\lim_{n \rightarrow \infty} \int_{0}^{n} \zeta_{\varepsilon, n}(\mu,t) \omega(\mu)d\mu=\int_{0}^{\infty} \zeta_{\varepsilon}(\mu,t) \omega(\mu)d\mu,
\end{align}
and as $\varepsilon\in(0,1]$
\begin{align}\label{Passlim3}
\lim_{\varepsilon \rightarrow 0}\int_{0}^{\infty} \zeta_{\varepsilon }(\mu,t) \omega(\mu) d\mu
=\int_{0}^{\infty} \zeta(\mu, t)\omega(\mu) d\mu.
\end{align}
In order to complete the proof it is sufficient to show that $\mathcal{P}^{1}_{\varepsilon, n}(t)\rightarrow\mathcal{P}^{1}(t)$ and $\mathcal{P}^{2}_{\varepsilon, n}(t)\rightarrow\mathcal{P}^{2}(t)$ as $n \rightarrow\infty$ and $\varepsilon \rightarrow 0$, where
\begin{align*}
\mathcal{P}^{1}(t)=\int_0^t \int_{0}^{\infty}\int_{0}^{\nu} \tau \omega' (\nu) \Lambda(\nu, \tau) \zeta(\nu, s) \zeta(\tau, s) d\tau d\nu ds,
\end{align*}
 and
\begin{align*}
\mathcal{P}^{2}(t)= \int_0^t \int_{0}^{\infty}\int_{0}^{\nu}  \omega(\tau) \Lambda(\nu, \tau) \zeta(\nu, s) \zeta(\tau, s) d\tau d\nu ds.
\end{align*}
Let us first split $ \mathcal{P}^{1}_{\varepsilon, n}$ into two sub-integrals as
\begin{align*}
&\mathcal{P}^{11}_{\varepsilon, n}(t)=\int_{0}^{t}\int_{0}^{\lambda}\int_{0}^{\nu}\bigg\{\frac{\omega(\nu+\varepsilon
\tau)-\omega(\nu)}{\varepsilon}\bigg\}   \Lambda(\nu, \tau) \zeta_{\varepsilon, n} (\tau, s) \zeta_{\varepsilon, n}(\nu,s) d\tau d\nu ds,\\
&\mathcal{P}^{12}_{\varepsilon, n}(t)=\int_{0}^{t}\int_{\lambda}^{n}\int_{0}^{\nu}\bigg\{\frac{\omega(\nu+\varepsilon \tau)-\omega(\nu)}{\varepsilon}\bigg\}   \Lambda(\nu, \tau) \zeta_{\varepsilon, n} (\tau, s) \zeta_{\varepsilon, n}(\nu,s) d\tau d\nu ds.
\end{align*}
Before passing  the limit in $\mathcal{P}^{11}_{\varepsilon,\lambda}$,  we  claim that for all $\nu\in (0,\lambda)$,
 \begin{align}\label{convergence1}
\lim_{n \rightarrow \infty} \int_{0}^{\nu}  \nu^{\sigma} \tau^{\sigma} & \bigg\{\frac{\omega(\nu+\varepsilon \tau)-\omega(\nu)}{\varepsilon}\bigg\}   \Lambda(\nu, \tau) \zeta_{\varepsilon, n} (\tau, s)  d\tau \nonumber\\
= & \int_{0}^{\nu}  \nu^{\sigma} \tau^{\sigma}  \bigg\{\frac{\omega(\nu+\varepsilon \tau)-\omega(\nu)}{\varepsilon}\bigg\}   \Lambda(\nu, \tau) \zeta_{\varepsilon}(\tau, s)  d\tau.
\end{align}

One can infer from \eqref{coagulation kernel} and $\omega \in W^{1, \infty}$ that $  \nu^{\sigma} \tau^{\sigma}  \bigg\{\frac{\omega(\nu+\varepsilon \tau)-\omega(\nu)}{\varepsilon}\bigg\}   \Lambda(\nu, \tau)\in L^{\infty}((0, \lambda)^{2})$.

Since $\zeta_{\varepsilon, n} \to \zeta_{\varepsilon}$ in $\mathcal{C}([0,T]_w; L^1(\mathbb{R}_{>0};  \mu^{-\sigma} d\mu))$, then by \cite[Lemma 4.3]{Laurencot:2002L}, we have
 \begin{align}\label{convergence3}
 \lim_{n \rightarrow \infty} \int_{0}^{\nu}  \nu^{\sigma}  & \bigg\{\frac{\omega(\nu+\varepsilon \tau)-\omega(\nu)}{\varepsilon}\bigg\}   \Lambda(\nu, \tau) \zeta_{\varepsilon, n} (\tau, s)  d\tau \nonumber\\
= & \int_{0}^{\nu}  \nu^{\sigma}   \bigg\{\frac{\omega(\nu+\varepsilon \tau)-\omega(\nu)}{\varepsilon}\bigg\}   \Lambda(\nu, \tau) \zeta_{\varepsilon} (\tau, s)  d\tau.
\end{align}
Next, from \eqref{coagulation kernel} and \eqref{Uniformboundtruncated}, the following integral can easily be shown finite i.e.\
\begin{align}\label{convergence2}
\int_{0}^{\nu}  \nu^{\sigma} \tau^{\sigma}  \bigg\{\frac{\omega(\nu+\varepsilon \tau)-\omega(\nu)}{\varepsilon}\bigg\}   \Lambda(\nu, \tau) \zeta_{\varepsilon, n} (\tau, s)  d\tau \le K(\lambda)  \|\omega\|_{ W^{1, \infty}(\mathbb{R}_{>0}) } \Theta <\infty,
\end{align}
where $K(\lambda)$ is a constant depending on $\lambda$. Applying once more \cite[Lemma 4.3]{Laurencot:2002L}, we obtain
 \begin{align}\label{convergence4}
 \lim_{n \rightarrow \infty} \int_{0}^{t} & \int_{0}^{\lambda}  \int_{0}^{\nu}     \bigg\{\frac{\omega(\nu+\varepsilon \tau)-\omega(\nu)}{\varepsilon}\bigg\}   \Lambda(\nu, \tau) \zeta_{\varepsilon, n} (\tau, s) \zeta_{\varepsilon, n} (\nu, s)  d\tau d\nu ds \nonumber\\
= & \int_{0}^{t} \int_{0}^{\lambda} \int_{0}^{\nu}   \bigg\{\frac{\omega(\nu+\varepsilon \tau)-\omega(\nu)}{\varepsilon}\bigg\}   \Lambda(\nu, \tau) \zeta_{\varepsilon} (\tau, s) \zeta_{\varepsilon} (\tau, s) d\tau d\nu ds.
\end{align}
As $\varepsilon \to 0$ to \eqref{convergence4}, we have
 \begin{align}\label{convergence5}
\lim_{\varepsilon \to 0}\lim_{n \to \infty} \mathcal{P}^{11}_{\varepsilon, n}(t)  =& \lim_{\varepsilon \to 0} \int_{0}^{t}  \int_{0}^{\lambda}  \int_{0}^{\nu}     \bigg\{\frac{\omega(\nu+\varepsilon \tau)-\omega(\nu)}{\varepsilon}\bigg\}   \Lambda(\nu, \tau) \zeta_{\varepsilon} (\tau, s) \zeta_{\varepsilon} (\nu, s)  d\tau d\nu ds \nonumber\\
= & \int_{0}^{t} \int_{0}^{\lambda} \int_{0}^{\nu} \tau  \omega'(\nu)   \Lambda(\nu, \tau) \zeta(\tau, s) \zeta(\tau, s) d\tau d\nu ds.
\end{align}

Finally, by using \eqref{coagulation kernel}, \eqref{Uniformboundtruncated}, \eqref{LemmalargevalueA}, the integrability properties of $\zeta$ from Definition~\ref{definition} and Lebesgue's dominated convergence theorem, we estimate that
\begin{align*}
\lim_{\lambda \to \infty} \mathcal{P}^{12}_{\varepsilon, n}(t)  = & \lim_{\lambda \to \infty} \int_{0}^{t}\int_{\lambda}^{n}\int_{0}^{\nu}\bigg\{\frac{\omega(\nu+\varepsilon \tau)-\omega(\nu)}{\varepsilon}\bigg\}   \Lambda(\nu, \tau) \zeta_{\varepsilon, n} (\tau, s) \zeta_{\varepsilon, n}(\nu,s) d\tau d\nu ds\\
  \le & k   \|  \omega \|_{W^{1, \infty} (\mathbb{R}_{>0})} \lim_{\lambda \to \infty} \int_{0}^{t}\int_{\lambda}^{n}\int_{0}^{1}   \nu \tau^{-\sigma} \zeta_{\varepsilon, n} (\tau, s) \zeta_{\varepsilon, n}(\nu,s) d\tau d\nu ds \\
   & + k   \|  \omega \|_{W^{1, \infty} (\mathbb{R}_{>0})} \lim_{\lambda \to \infty} \int_{0}^{t}\int_{\lambda}^{n}\int_{1}^{\nu}   (\nu+\tau) \zeta_{\varepsilon, n} (\tau, s) \zeta_{\varepsilon, n}(\nu,s) d\tau d\nu ds \\
   \le & 3k   \|  \omega \|_{W^{1, \infty} (\mathbb{R}_{>0})} \lim_{\lambda \to \infty} \frac{\lambda} {\Psi(\lambda)} \Theta \int_{0}^{t}\int_{\lambda}^{n}  \Psi(\nu)  \zeta_{\varepsilon, n}(\nu,s) d\nu ds \\
    \le & 3k    \|  \omega \|_{W^{1, \infty}(\mathbb{R}_{>0})}  \lim_{\lambda \to \infty} \frac{\lambda} {\Psi(\lambda)} \Theta \Theta_2(T)T = 0.
\end{align*}

Similarly, we can prove that for $\varepsilon \to 0$,
 \begin{align}\label{convergence6}
\lim_{\lambda \to \infty}  \mathcal{P}^{12}_{\varepsilon}(t)
=  \lim_{\lambda \to \infty} \int_{0}^{t}\int_{\lambda}^{\infty}\int_{0}^{\nu}  \tau \omega'(\nu)   \Lambda(\nu, \tau) \zeta_{\varepsilon} (\tau, s) \zeta_{\varepsilon}(\nu,s) d\tau d\nu ds=0.
\end{align}

Combing \eqref{convergence5} and \eqref{convergence6}, we obtain
\begin{align}
\lim_{\varepsilon \to 0 } \lim_{n \to \infty} \mathcal{P}^{1}_{\varepsilon, n}(t) = \mathcal{P}^{1}(t).
\end{align}
Similarly, it can be easily seen that
\begin{align}
\lim_{\varepsilon \to 0 } \lim_{n \to \infty} \mathcal{P}^{2}_{\varepsilon, n}(t) = \mathcal{P}^{2}(t).
\end{align}
This completes the weak convergence of integral operators. Hence, $\zeta$ is a weak solution to \eqref{OHS1}--\eqref{SCEin}.

\section{Weak solutions are mass-conserving}

For the purpose of completing the proof of Theorem~\ref{Theorem1}, it is sufficient to show that all solutions are mass-conserving. For this, we need the following lemma which will give the sketch of the proof of the equation \ref{ThemMasscons}.
\begin{lemma}\label{lemma1}
 Suppose $\Lambda$ satisfies \eqref{coagulation kernel}--\eqref{bound coagulation kernel} and $\zeta^{in} \in \mathcal{Y}^+$. Let $\zeta$ be a weak solution to \eqref{OHS1}--\eqref{SCEin} on $[0, T)$. Then, we have
\begin{align}\label{lemma11}
\int_0^{\lambda} \mu \zeta(\mu, t)  d\mu -\int_0^{\lambda} \mu \zeta^{in}(\mu )  d\mu =
-\int_{0}^{t} \int_{\lambda}^{\infty}  \int_0^{\lambda}  \nu \Lambda(\mu, \nu) \zeta(\mu, s) \zeta(\nu, s) d\nu d\mu ds,
\end{align}
for $\lambda \in \mathbb{R}_{>0}$ and $t\in (0,T)$.
\end{lemma}

\begin{proof}
Set  $\omega (\mu)=\mu \chi_{(0, \lambda)}(\mu)$ for $\mu \in \mathbb{R}_{>0}$ and substitute this into \eqref{Identity4} to have
\begin{equation*}
 \omega_{\varepsilon}(\mu, \nu)=\begin{cases}
0,\ & \text{if}\ (\mu, \nu ) \in (0, \lambda) \times (0, \mu),\ \\
- \nu, \  &  \text{if}\ (\mu,\nu) \in [\lambda, \infty) \times (0, \lambda),\ \\
0,\ &  \text{if}\ (\mu, \nu)\in [\lambda, \infty) \times [\lambda, \mu).
\end{cases}
\end{equation*}
Inserting the above values of $\omega_{\varepsilon}$ into \eqref{WeakGenSCEli}, we find
\begin{align*}
\int_0^{\lambda}[ \zeta({\mu}, t) - \zeta^{in}({\mu})] {\mu} d{\mu}
=&{-\int_0^t  \int_{\lambda}^\infty \int_0^{\lambda}  \nu \Lambda({\mu},{\nu}) \zeta({\mu}, s) \zeta({\nu}, s)d{\nu}d{\mu}ds},
\end{align*}
which settles the proof of Lemma~\ref{lemma1}.
\end{proof}

In order to complete the proof of equation \eqref{ThemMasscons}, it is enough to show that the right-hand side of (\ref{lemma11}) goes
to zero as $\lambda \to \infty$.

\begin{lemma}\label{lemma2}
Assume that $\Lambda$ satisfies \eqref{coagulation kernel} and $\zeta^{in} \in \mathcal{Y}^+$. Let the weak solution $\zeta$ to
\eqref{OHS1}--\eqref{SCEin}. Then the following holds
\begin{align*}
 \lim_{\lambda \to \infty}\int_0^t \int_{\lambda}^{\infty} \int_0^{\lambda}  \nu \Lambda({\mu}, {\nu})  \zeta({\mu}, s)  \zeta({\nu}, s)d{\nu}d{\mu} ds=0,
\end{align*}
for each $t\in (0,T)$.
\end{lemma}

\begin{proof}
Assume $\lambda >1$, $t\in (0,T)$, and $s\in (0, t)$. In order to prove Lemma~\ref{lemma2}, we first split the following integral into two sub-integrals as
\begin{equation*}
\int_{\lambda}^{\infty} \int_0^{\lambda}  \nu \Lambda({\mu},{\nu}) \zeta({\mu}, s) \zeta({\nu}, s) d{\nu} d{\mu} = \Sigma_1(\lambda, s) + \Sigma_2(\lambda, s),
\end{equation*}
with
\begin{align*}
\Sigma_1(\lambda, s) & := \int_{\lambda}^{\infty} \int_0^1  \nu \Lambda({\mu},{\nu}) \zeta({\mu}, s) \zeta({\nu}, s) d{\nu} d{\mu}, \\
\Sigma_2(\lambda, s) & := \int_{\lambda}^{\infty} \int_1^{\lambda}  \nu \Lambda({\mu},{\nu}) \zeta({\mu}, s) \zeta({\nu}, s) d{\nu} d{\mu}.
\end{align*}
On the one hand, we infer from \eqref{coagulation kernel} and Young's inequality that
\begin{align*}
\Sigma_1(\lambda, s)  & \le k\int_{\lambda}^{\infty} \int_0^1    \nu^{1-\sigma} \mu \zeta({\mu}, s) \zeta({\nu}, s) d{\nu} d{\mu} \\
& \le k \left( \int_0^\infty \nu^{1/2-\sigma}  \zeta(\nu, s) d\nu \right) \left( \int_{\lambda}^\infty \mu \zeta(\mu, s) d\mu \right)  \le k \|\zeta(s)\|_{ \mathcal{Y}} \int_{\lambda}^\infty \mu \zeta(\mu,s) d\mu,
\end{align*}
and by using the integrability properties of $\zeta$ from Definition~\ref{definition} and Lebesgue's dominated convergence theorem, we obtain
\begin{equation}\label{Barik1}
\lim_{\lambda \to \infty}\int_0^t \Sigma_1(\lambda, s) ds = 0.
\end{equation}
On the other hand, we infer from \eqref{coagulation kernel} that
\begin{align*}
\Sigma_2(\lambda, s) & \le k  \int_{\lambda}^{\infty}  \int_1^{\lambda} \nu (\mu+\nu) \zeta({\mu},s) \zeta({\nu},s) d{\nu} d{\mu} \\
& \le 2k  \int_{\lambda}^{\infty} \int_1^{\lambda}  \mu \nu \zeta({\mu},s) \zeta({\nu},s) d{\nu} d{\mu}  \le 2k \|\zeta(s)\|_{\mathcal{Y}} \int_{\lambda}^\infty \mu \zeta(\mu,s) d\mu,
\end{align*}
and using the same argument as in \eqref{Barik1}, we conclude that
\begin{equation*}
\lim_{\lambda\to \infty}\int_0^t \Sigma_2(\lambda, s) ds = 0.
\end{equation*}
Recalling \eqref{Barik1}, we  obtain the desired result of Lemma~\ref{lemma2}.\\
\end{proof}
Now, we are in a position to prove \eqref{ThemMasscons}.
\begin{proof}[Proof of \eqref{ThemMasscons}]  Let $t\in (0,T)$. Then, from Lemma~\ref{lemma2}, we find
\begin{equation}\label{mce10}
\lim_{\lambda \to \infty}\int_0^t \int_{\lambda}^{\infty} \int_0^{\lambda} \nu \Lambda({\mu},{\nu}) \zeta({\mu},s) \zeta({\nu},s)d{\nu}d{\mu} ds=0,
\end{equation}
it readily follows from \eqref{mce10} that the left-hand side of \eqref{lemma11} converges to zero as $\lambda \to\infty$. Thus, we have
\begin{equation*}
\mathcal{M}_1(\zeta)(t) = \lim_{\lambda \to\infty} \int_0^{\lambda} \mu \zeta(\mu, t) d\mu = \lim_{\lambda \to\infty} \int_0^{\lambda} \mu \zeta^{in}(\mu) d\mu
= \mathcal{M}_1(\zeta^{in}).
\end{equation*}
This completes the proof of \eqref{ThemMasscons}.
\end{proof}

\section*{Acknowledgments}
AKG would like to thank Philippe Lauren\c{c}ot, University of Toulouse, CNRS, Toulouse for valuable discussions that have helped to improve the content of the manuscript.




\begin{thebibliography}{11}
\small{

\bibitem{Aldous:1999} D. J. Aldous, Deterministic and stochastic model for coalescence (aggregation and coagulation): a review of the mean-field theory for probabilists, \textit{Bernoulli}, \textbf{5} 3--48, 1999.

\bibitem{Bagland:2005} V. Bagland, Convergence of a discrete Oort-Hulst-Safronov equation, \textit{Math. Methods Appl. Sci.}, \textbf{28} 1613--1632, 2005.


\bibitem{Bagland:2007} V. Bagland and Ph. Lauren\c{c}ot, Self-similar solutions to the Oort-Hulst-Safronov coagulation equation, \textit{SIAM J. Math. Anal.}, \textbf{39} 345--378, 2007.


\bibitem{Ball:1990} J. Ball and J. Carr, The discrete coagulation-fragmentation equations: Existence, uniqueness and density conservation,
\textit{J. Statist. Phys.}, \textbf{61} 203--234, 1990.






\bibitem{Barik:2018} P. K. Barik and A. K. Giri, A note on mass-conserving solutions to the coagulation-fragmentation equation by using non-conservative approximation, \textit{Kinet. Relat. Models}, \textbf{11} 1125--1138, 2018.



\bibitem{Barik:2019} P. K. Barik, A. K. Giri and Ph. Lauren\c{c}ot, Mass-conserving solutions to Smoluchowski coagulation equation with singular kernel,  \textit{Proc. Roy. Soc. Edinb. Sec. A: Math.}, DOI: https://doi.org/10.1017/prm.2018.158.




\bibitem{Camejo:2015} C. C. Camejo, R. Gr\"{o}pler and G. Warnecke, Regular solutions to the coagulation equations with singular kernels, \textit{Math. Methods Appl. Sci.}, \textbf{38} 2171--2184. 2015.





\bibitem{Davidson:2014} J. Davidson, Existence and uniqueness theorem for the Safronov-Dubovski coagulation equation, \textit{Z. Angew.
Math. Phys.}, \textbf{65} 757--766, 2014.


\bibitem{Dubovski:1999I} P. B. Dubovski, A triangle of interconnected coagulation models, \textit{J. Phys. A. Math. Gen.}, \textbf{32} 781--793, 1999.


\bibitem{Dubovski:1999II} P. B. Dubovski, Structural stability of disperse systems and finite nature of the coagulation front, \textit{J. Exp. Theor.
Phys.}, \textbf{89} 384--390, 1999.

\bibitem{Escobedo:2002} M. Escobedo, S. Mischler and B. Perthame, Gelation in coagulation and fragmentation models, \textit{Comm. Math. Phys.}, \textbf{231} 157--188, 2002.

\bibitem{Escobedo:2003} M. Escobedo, Ph. Lauren\c{c}ot,  S. Mischler and B. Perthame, Gelation and mass conservation in coagulation-fragmentation models,
\textit{J. Diff. Eqs.}, \textbf{195} 143--174, 2003.


\bibitem{Escobedo:2006} M. Escobedo and S. Mischler, Dust and self-similarity for the Smoluchowski coagulation equation, \textit{Ann. Inst. H. Poincar\'e Anal. non Lin\'eaire}, \textbf{23} 331--362, 2006.



\bibitem{Giri:2011}  A. K. Giri, J. Kumar and G. Warnecke, The continuous coagulation equation with multiple fragmentation, \textit{J. Math. Anal. Appl.}, \textbf{374} 71--87, 2011.

\bibitem{Giri:2012} A. K. Giri, Ph. Lauren\c{c}ot and G. Warnecke, Weak solutions to the continuous coagulation with multiple fragmentation, \textit{Nonlinear Anal.}, \textbf{75}, 2199--2208, 2012.



\bibitem{Lachowicz:2003} M. Lachowicz, Ph. Lauren\c{c}ot and D. Wrzosek, On the Oort-Hulst-Safronov coagulation equation and its relation to the Smoluchoski equation, \textit{SIAM J. Math. Anal.}, \textbf{34} 1399--1421, 2003.



\bibitem{Laurencot:2005} Ph. Lauren\c{c}ot, Convergence to self-similar solutions for a coagulation equation, \textit{Z. Angew. Math. Phys.}, \textbf{56} 398--411, 2005.



\bibitem{Laurencot:2006} Ph. Lauren\c{c}ot,  Self-similar solutions to a coagulation equation with multiplicative kernel, \textit{Physica D}, \textbf{222} 80--87, 2006.

\bibitem{Laurencot:2015} Ph. Lauren\c{c}ot,
\textrm{ Weak compactness techniques and coagulation equations},
 in \textit{Evolutionary Equations with Applications in Natural Sciences}, J.~Banasiak \& M.~Mokhtar-Kharroubi (eds.), Lecture Notes Math. \textbf{2126} Springer, 199--253, 2015.

\bibitem{Laurencot:2002L} Ph. Lauren\c{c}ot and S. Mischler, From the discrete to the continuous coagulation-fragmentation equations,
\textit{Proc. Roy. Soc. Edinburgh}, \textbf{132A} 1219--1248, 2002.

\bibitem{Mischler:2004} Ph. Lauren\c{c}ot and S. Mischler, On coalescence equations and related models, in \textit{Modeling and Computational Methods for Kinetic Equations}, Model. Simul. Sci. Eng. Technol., Birkha\"user, Boston, 321--356, 2004.


\bibitem{Leyvraz:1981} F. Leyvraz and H. R. Tschudi, Singularities in the kinetics of coagulation processes,
\textit{J. Phys. A}, \textbf{14} 3389--3405, 1981.

\bibitem{Muller:1928}  H. M\"{u}ller, Zur allgemeinen Theorie der raschen Koagulation,
\textit{Kolloidchemische Beihefte}, \textbf{27} 223--250, 1928.



\bibitem{Norris:1999}
\newblock J. R. Norris, Smoluchowski's coagulation equation: uniqueness, non-uniqueness and hydrodynamic limit for the stochastic coalescent,
 \textit{Ann. Appl. Probab.}, \textbf{9} 78--109, 1999.



\bibitem{Hulst:1946} J. H. Oort and H. C. van de Hulst, Gas and smoke in interstellar space, \textit{Bulletin of the Astronomical Institute of the
Netherlands}, \textbf{10} 187--210, 1946.



\bibitem{Safronov:1972} V. S. Safronov, \textit{Evolution of the Protoplanetary Cloud and Formation of the Earth and the Planets}, Israel
Program for Scientific Translations, Jerusalem, 1972.



\bibitem{Smoluchowski:1917}
 M.~Smoluchowski, Versuch einer mathematischen Theorie der Koagulationskinetik kolloider L\"osungen,
\textit{Zeitschrift f. physik. Chemie}, \textbf{92} 129--168, 1917.



\bibitem{Stewart:1989} I. W. Stewart, A global existence theorem for the general coagulation-fragmentation equation with unbounded kernels,
\textit{Math. Methods Appl. Sci.}, \textbf{11} 627--648, 1989.



\bibitem{Stewart:1990} I. W. Stewart, A uniqueness theorem for the coagulation-fragmentation equation,
\textit{Math. Proc. Camb. Phil. Soc.}, \textbf{107} 573--578, 1990.



\bibitem{Vrabie:1995}
I. I. Vrabie, \textit{Compactness Methods for Nonlinear Evolutions}, 2nd edition,
 Pitman Monogr. Surveys Pure Appl. Math., \textbf{75} Longman, 1995.

}

\end{thebibliography}
\end{document}